\documentclass{amsart}

\usepackage{bm}
\usepackage{amsthm}
\usepackage{amsmath}
\usepackage{amssymb}
\usepackage{enumerate} 
\usepackage{comment}
\usepackage{mathtools}
\usepackage{url}
\usepackage{color}

\theoremstyle{definition}
\newtheorem{dfn}{Definition}[section]
\newtheorem{prop}[dfn]{Proposition}

\newtheorem{lem}[dfn]{Lemma}

\newtheorem{thm}[dfn]{Theorem}
\newtheorem{cor}[dfn]{Corollary}
\newtheorem{rem}[dfn]{Remark}

\newtheorem{que}[dfn]{Question}

\newcommand{\1}{
    \bm{1}
}

\newcommand{\res}[1]{
(\mathbb{Z}/{ #1 \mathbb{Z}})^{\times}
}

\newcommand{\lep}{
\lesssim
}

\title{Small gaps between Goldbach primes}

\author{Mizuki Akeno}
\date{\today}

\email{akeno.mizuki.tkb\_ee@u.tsukuba.ac.jp}
\date{\today}
\address{College of Mathematics, University of Tsukuba, Tsukuba, Japan}
\subjclass[2020]{}

\begin{document}

\begin{abstract}
We study small gaps between Goldbach primes $\mathbb{P} \cap (N-\mathbb{P})$ using the Bombieri-Davenport method and the Maynard-Tao method, and compare the two. 

We show that for almost all even integers $N$, the smallest gap in $\mathbb{P} \cap (N-\mathbb{P})$ is at most $0.765\ldots$ times the average gap, using the Bombieri-Davenport method. This improves a recent result of Tsuda. We also demonstrate that a straightforward application of the Maynard-Tao method is insufficient to improve this bound. However, it allows us to establish the existence of bounded gaps between Goldbach primes with bounded error for almost all even integers $N$. 
\end{abstract}

\maketitle

\section{Introduction}

Let $\mathbb{P}$ denote the set of primes. Bombieri and Davenport \cite{bombieri1966small} showed that 
$$  \min_{\substack{p,p' \in \mathbb{P} \\ X < p'<p \leq 2X}} (p-p') \leq (\tfrac{1}{2} + \varepsilon) (\ln{X}) $$
for any $\varepsilon>0$ and sufficiently large $X>0$. They used the Bombieri-Vinogradov theorem and the circle method. 
It was pointed out in \cite{goldston2007higher3} that their argument is essentially equivalent to analyzing the second moment 
\begin{equation} \label{BD start} \sum_{X<n\leq 2X} \left( \sum_{h \leq H} (\Lambda(n+h) - \Lambda^{\sharp}(n+h)) \right)^2, \end{equation}
where $H \asymp \ln{X}$ and $\Lambda^{\sharp}$ is a suitable approximation to the von Mangoldt function $\Lambda$. 

Later, building on the work of Goldston-Pintz-Yildirim \cite{goldston2009primes}, Zhang, Maynard \cite{maynard2015small} and Tao (unpublished), it was shown in the Polymath project \cite{polymath2014variants} that
$$  \min_{\substack{p,p' \in \mathbb{P} \\ X < p'<p \leq 2X}} (p-p') \leq 246 $$
for all sufficiently large $X>0$. They considered a weighted sum:
\begin{equation} \label{MT start} \sum_{X<n\leq 2X} \left( \sum_{h \in \mathcal{H}} \Lambda(n+h) - \ln{3X} \right) w(n;\mathcal{H}), \end{equation}
where $\mathcal{H}$ is a finite set of integers and $w(n;\mathcal{H})$ is a suitable non-negative weight.

For a positive integer $N$, define 
$$ \mathbb{P}(N) = \{ p \leq N: \exists p' \in \mathbb{P}, p+p'=N \} = \mathbb{P} \cap (N - \mathbb{P}). $$
It is well known that
\begin{equation} \label{|PN|} |\mathbb{P}(N)| = \frac{N}{(\ln{N})^2} \mathfrak{S}(N) (1 +o(1)) \end{equation}
holds for almost all even integers $N$. Here, $\mathfrak{S}(N)$ is called the singular series and is defined by
$$ \mathfrak{S}(N) = \prod_{p} \left(1 - \frac{|\{0,N\}/p\mathbb{Z}|}{p} \right) \left( 1 - \frac{1}{p} \right)^{-2}, $$
and ``almost all'' means that for any $X>0$, the number of even integers $N \leq X$ that do not satisfy \eqref{|PN|} is at most $o(X)$.

\begin{dfn}
Let $\Xi \geq 0$ be the infimum value for which the following statement holds. 

For any given $\epsilon, \varepsilon>0$, 
$$ \min_{\substack{p,p' \in \mathbb{P}(N) \\ N^{1-\epsilon} < p'<p}} (p-p') \leq (\Xi + \varepsilon) \mathfrak{S}(N)^{-1} (\ln{N})^2 $$
holds for almost all even integers $N$. 

Similarly, let  $\Xi^{*} \geq 0$ be the infimum value for which the following statement holds, if it exists. 
For any given $\epsilon>0$, 
$$ \min_{\substack{p,p' \in \mathbb{P}(N) \\ N^{1-\epsilon} < p'<p}} (p-p') \leq \Xi^{*} $$
holds for almost all even integers $N$. 
\end{dfn}

By \eqref{|PN|} and the pigeonhole argument, $\Xi \leq 1$. Mikawa and Nakamura essentially proved $\Xi \leq \frac{5}{6} = 0.8333\ldots$ in unpublished work. Furthermore, Tsuda's result \cite{tsuda2024small} shows that $\Xi \leq 0.8201\ldots$. These results are based on the Bombieri-Davenport method. The generalized Hardy-Littlewood conjecture implies that $\Xi^{*}=6$ and thus $\Xi=0$. We hope to show that $\Xi=0$, or ultimately to establish the existence of $\Xi^{*}$.

The main aim of this paper is to examine the Maynard-Tao method when applied to this problem and to compare it with the Bombieri-Davenport method. 
We consider two types of sums:
\[ \sum_{m+n=N} \left( \sum_{h \leq H} (\Lambda(m+h) \Lambda(n-h)  - (\Lambda(m+h) \Lambda(n-h))^{\sharp}) \right)^2, \]
\[ \sum_{m+n=N} \left( \sum_{h \in \mathcal{H}} \Lambda(m+h) \Lambda(n-h) - (\ln{N})^2 \right) w(m,n;\mathcal{H}) , \]
as analogues of \eqref{BD start} and \eqref{MT start}, respectively. 

As a result, we show that
\begin{thm}\label{Xi<}
$$ \Xi \leq \frac{7}{72} (4+\sqrt{15}) = 0.76542\ldots $$
\end{thm}
by the Bombieri-Davenport method. This improvement comes from a simple observation that refines the level of distribution of $\mathbb{P}(N)$. 

We demonstrate that a straightforward application of the Maynard-Tao method, involving a form similar to \eqref{MT start}, is insufficient to prove the existence of $\Xi^{*}$, and cannot even be used to improve the above bound on $\Xi$. The difficulty lies in finding an optimal weight $w$. Indeed, we observe the limitation of the variational problem and see that one would not derive a better bound than $\Xi \leq 0.826\ldots$ using the current method. 
However, the Maynard-Tao method is sufficient to show the existence of bounded gaps between primes that satisfy the Goldbach equation with a bounded error. 

For any $H \geq 0$ and an integer $N$, let
$$ \mathbb{P}_H(N) = \{ p \leq N: \exists p' \in \mathbb{P}, |N-p-p'| \leq H \}. $$
\begin{thm}\label{Xi**<}
There exists $H>0$ such that the following holds. 
For any given $\epsilon>0$, 
$$ \min_{\substack{p,p' \in \mathbb{P}_H(N) \\ N^{1-\epsilon} < p'<p}} (p-p') \leq H $$
holds for almost all even integers $N$. 
One can take $H=56250000=(7500)^2$. 
\end{thm}
We remark that the size of the exceptional set in all the above results is $O(X (\ln{X})^{-A})$ for any $A>0$.  

Matom\"aki and Shao showed in Theorem 1.1 of  \cite{matomaki2017vinogradov} that for every $m>0$ there exists $H(m)>0$ such that every sufficiently large odd integer $N$ can be written in the form $N = p_1+p_2+p_3$, 
where, for $i = 1, 2, 3$, each $p_i$ is a prime such that the interval $[p_i, p_i + H(m)]$ contains at least $m$ primes. 
Their argument relies on the transference principle \footnote{See also \cite{green2006restriction} and \cite{matomaki2017vinogradov_} for other deep versions of the transference principle.}. 
It is conceivable that a result similar to Theorem \ref{Xi**<} could also be obtained by using either the transference principle, or the circle method combined with a non-negative model \cite{grimmelt2022goldbach}. 
Our argument, however, is more direct and yields a sharper bound on $H$. 
Both of their approaches require $k=|\mathcal{H}|$ large enough to meet the variational criterion in \cite{maynard2025primes} with $\theta=1/3$. By contrast, our method requires only the analogous criterion $\theta = 1/\sqrt{6}$; it is plausible that weaker assumptions would also be sufficient. 

We note that our bounds on $H$ are not numerically optimized; we did not employ quadratic optimization to determine the sieve weights, nor did we use any computer-assisted search for narrow admissible tuples.

We also note that one may obtain Theorem 1.1 in \cite{matomaki2017vinogradov} with an explicit bound on $H(m)$ by a method similar to ours, namely by calculating
$$ \sum_{n_1+n_2+n_3=N} \left( \sum_{h_i \in \mathcal{H}_i} \1_{\mathbb{P}}(n_1+h_1) \1_{\mathbb{P}}(n_2+h_2) \1_{\mathbb{P}}(n_3+h_3) - (m-1) \right) w(n_1,n_2,n_3;\mathcal{H}_1,\mathcal{H}_2,\mathcal{H}_3) $$
for some tuples $\mathcal{H}_i$ and a non-negative weight $w$.

\subsection{Remark on the original problem}

Mikawa-Nakamura and Tsuda considered the above problem without imposing size restrictions on $p,p'$. Their results concern a bound on the constant $\Xi'$, which is defined as the infimum for which
$$ \min_{p,p' \in \mathbb{P}(N)} (p-p') \leq (\Xi' + \varepsilon) \mathfrak{S}(N)^{-1} (\ln{N})^2 $$
holds for any $\varepsilon>0$ and almost all even integers $N$. Heuristically, the probability that a given $n$ lies in $\mathbb{P}(N)$ is approximately given by $\frac{1}{(\ln{n})(\ln{(N-n)})}$. Thus we expect that $\mathbb{P}(N)$ is denser near the edges of $[0,N]$. It therefore seems more natural to impose size restrictions on $p,p'$ so that $\frac{1}{(\ln{p}) (\ln{p'})} \approx \frac{1}{(\ln{N})^2}$ in this problem. 

A Siegel-Walfisz type theorem for almost all short intervals yields the following result. 
\begin{prop}
Let $c \in (1/6,1]$. For almost all even integers $N \in (X/2,X]$, we have
$$ \sum_{\substack{p+p'=N \\ p \leq X^c}} 1 = \frac{\mathfrak{S}(N)X^c}{(\ln{X^c}) (\ln{N})} (1+o(1)). $$
Furthermore, we have
\begin{equation*} \Xi' \leq \frac{1}{6}. \end{equation*}
\end{prop}
We omit the proof. The bound for $\Xi'$ follows from a pigeonhole argument. Although this may be a folklore result, we have not found a written proof or published reference. We also remark that the generalized Riemann hypothesis implies $\Xi'=0$. 

Therefore, we impose the restriction $N^{1-\epsilon} < p,p' $ for any fixed $\epsilon>0$ throughout. One may instead consider the restriction $\frac{1}{3} N \leq p,p' \leq \frac{2}{3} N$, though it would not affect the conclusion and would be slightly more cumbersome. 

\subsection{Organization}
In Section \ref{LV}, we provide variants of level of distribution of $ \{ (p,p') : p+p'=N \} $ and $ \{ (m,n) : m+n=N \} $. 
We follow the arguments of Mikawa-Nakamura and Tsuda in a general setting with some simplifications.

In Section \ref{ML}, we prove two main lemmas: asymptotic formulas for the correlation of $\Lambda_Q$ and asymptotic formulas for the sum of singular series. These results play an  important role in both settings: the GPY/Maynard-Tao method and the Bombieri-Davenport method.

In Section \ref{BD}, we apply a simplified version of the Bombieri-Davenport method due to Goldston and Y{\i}ld{\i}r{\i}m \cite{goldston2007higher3}, and prove Theorem \ref{Xi<}.

In Section \ref{MT}, we consider the GPY/Maynard-Tao method. To prove an asymptotic formula for the sum of sieve weights, we adopt a method different from both Maynard's and Tao's. We examine the limitations of this method for bounding $\Xi$ and prove Theorem \ref{Xi**<}.

Some of the intermediate results, such as Lemma \ref{gal local} and the notion
of Goldbach-admissible tuples introduced in Section \ref{MT}, may also be of
independent interest.

\subsection{Notation}

We use the following standard notation throughout the paper:
\begin{itemize}
\item $\sigma_a(n)$ denotes the sum of $a$-th powers of the divisors of $n$, and $\tau(n)$ denotes $\sigma_0(n)$. \par
\item $c_q(n)$ denotes the Ramanujan sum, $\varphi(q)$ denotes the Euler totient function (i.e., $c_q(0)$) and $\mu(n)$ denotes the Möbius function (i.e., $c_q(1)$). \par
\item $e(\alpha)$ denotes $e^{2\pi i \alpha}$, $\mathbb{T}$ denotes $\mathbb{R}/\mathbb{Z}$, and $\| x \|$ denotes the distance from $x$ to the nearest integer. \par
\item The notations $O$ and $o$ are Landau symbols. $A \ll B$ denotes Vinogradov's notation. We generally omit the parameter dependencies. We sometimes use a notation such as $f(x) = g(x)^{O(1)}$, which means there exists a function $h$ satisfying $h(x)=O(1)$ and $f(x) = g(x)^{h(x)}$. 
\item $\1_A(\cdot)$ denotes the indicator function of a set $A$ as usual, and $\1_{P}$ denotes the indicator function of a predicate $P$. 
\end{itemize}

We denote by $\Lambda$ the von Mangoldt function, extended to $\mathbb{Z}$ by setting $\Lambda(-n)=\Lambda(n)$ and $\Lambda(0)=0$ for convenience.

For $Q \geq 1$, we define $\Lambda_Q=\Lambda_{Q,\text{SEL}}$ by
$$ \Lambda_Q(n) = \sum_{\substack{q|n \\ q \leq Q}} \mu(q) \ln{Q/q}. $$
This is an approximation to $\Lambda(n)$ introduced by Selberg. Throughout this paper, $\Lambda_Q$ plays an important role. 
All the results shown in this paper can also be obtained by using $\Lambda_{Q,\text{HB}}$ instead of $\Lambda_{Q,\text{SEL}}$, where
$$ \Lambda_{Q,\text{HB}}(n) = \sum_{q \leq Q} \frac{\mu(q)}{\varphi(q)} c_q(n) $$
as introduced by Heath-Brown \cite{heath1985ternary}.

\section{Level of distribution}\label{LV}
We will use $X$ as a global parameter throughout this section. The notation $Y \lep Z$ means $Y \ll Z (\ln{X})^{B}$ for some absolute constant $B$.

The following theorem was essentially proved in an unpublished paper by Mikawa and Nakamura. They proved this statement for specific choices of $f_i$, constant functions $l_i$, and $\mathcal{D}=[0,1/2] \times [0,1/3]$. By reversing $f_1$ and $f_2$, the result also holds for $([0,1/2] \times [0,1/3]) \cup ([0,1/3] \times [0,1/2])$. The improvement in Theorem \ref{Xi<} comes from this simple observation.

\begin{thm}[Variants of level of distribution of $\{(p,p') : p+p'=N\}$]\label{LL level}
Let
$$ \mathcal{D} = ([0,1/2] \times [0,1/3]) \cup ([0,1/3] \times [0,1/2]). $$
Let $X \geq 2$ and $l_1,l_2:\mathbb{N} \to \mathbb{Z}$ be functions such that 
$$ \sum_{d \leq X} \frac{((d,l_1(d))+(d,l_2(d)))}{d} \ll X^{o(1)}. $$
Let $f_1,f_2$ be $1$-bounded arithmetic functions. For any $B>0$, the following holds:

Let $R=X^{\beta}$ with $\beta \in (0,1/2]$. For all integers $N \in (X/2,X]$ with $O(X (\ln{X})^{-B} )$ exceptions, we have
\begin{eqnarray} \label{level} \sum_{(d_1,d_2) \in X^{(1-\varepsilon) \cdot \mathcal{D}}}  f_1(d_1) f_2(d_2) E_R(l_1(d_1),l_2(d_2),d_1,d_2;N) \ll \frac{X}{(\ln{X})^B}
\end{eqnarray}
where 
$$ E_R(l_1,l_2,d_1,d_2;N) = \sum_{\substack{m+n=N \\ m \equiv l_1 \bmod d_1 \\ n \equiv l_2 \bmod d_2}} (\Lambda(m) \Lambda(n) - \Lambda_R(m) \Lambda_R(n)), $$
and
$$ X^{(1-\varepsilon) \cdot \mathcal{D}} = \{ (d_1,d_2) \in \mathbb{N}^2 : \left( \frac{\ln{d_1}}{\ln{X}}, \frac{\ln{d_2}}{\ln{X}} \right) \in (1-\varepsilon) \cdot \mathcal{D} \}. $$
\end{thm}

Note that we require $f_i$ and $l_i$ to be independent of $N$ in this theorem. 

\begin{rem}
If we replace Lemma \ref{mL} with Matom{\"a}ki's estimate for exponential sums (Corollary 3 in \cite{matomaki2009bombieri}), the same holds with $\mathcal{D}$ replaced by
\begin{equation} \label{D star} \mathcal{D}^* = [0,1/2] \times [0,1/2], \end{equation}
under the assumption that either $f_1$ or $f_2$ is well-factorable of level $X^{(1-\varepsilon)/2}$. 
This is essentially done in Tsuda \cite{tsuda2024small}. 
\end{rem}

\begin{rem}
By a slight modification of the Maier-Pomerance result concerning the distribution of generalized twin primes in arithmetic progressions (see Theorem 3.1, Section 6 of \cite{maier1990unusually}), one can establish the level of distribution of $\mathbb{P}(N)$ in the standard sense. One can show that there exists $c>0$ such that 
$$ \sum_{d \leq X^c} \sup_{l:(l,d)=(N-l,d)=1} \left| \sum_{\substack{n \in \mathbb{P}(N) \\ n \equiv l \bmod d}} (\ln{n})(\ln{(N-n)})  - \frac{N}{\varphi(d)} \sum_{q} \frac{\mu(q)^2}{\varphi(q)^2} c_q(dN) \right| \ll \frac{X}{(\ln{X})^A} $$
holds for all even integers $N \in (X/2,X]$ with $O(X (\ln{X})^{-B} )$ exceptions. 
Their argument is based on the Montgomery-Vaughan work \cite{montgomery1975exceptional} on the bound for the exceptional set in the Goldbach conjecture. 
We also remark that one can take $c < \frac{10}{314} = 0.0318...$ by combining it with standard zero-density estimate for Dirichlet $L$-functions (not necessarily ``log-free''). 
This suffices to prove Theorem \ref{Xi**<} by a method similar to that of Maynard \cite{maynard2015small}, although the resulting bound for $H$ and for the size of the exceptional set would be very weak. 
\end{rem}

\subsection{Preliminaries}

The next two lemmas are well known in analytic number theory. 
\begin{lem}\label{tau bound}
Let $k$ be a positive integer. Let $d \leq X^{1-\varepsilon}$ and $b$ be integers. We have the following. 
$$ \tau(n) \ll n^{o(1)} $$
$$ \sum_{n \leq X} \frac{\tau(n)^{k}}{n} \ll (\ln{X})^{O_k(1)} $$ 
$$ \sum_{m,n \leq X} \frac{\tau(m)^{k} \tau(n)^{k}}{[m,n]} \ll (\ln{X})^{O_k(1)} $$ 
$$ \sum_{\substack{n \leq X \\ n \equiv b \bmod d}} \tau(n)^{k} \ll \frac{X}{d} \tau(d)^k (\ln{X})^{O_{k.\varepsilon}(1)} $$ 
\end{lem}
\begin{proof}
The first and second estimates are well known. We omit the proof. For the third estimate, we use $(m,n) \leq \sum_{d|m,n} d$ and $\tau(dn') \leq \tau(d) \tau(n')$; this gives
\begin{eqnarray*}
\sum_{m,n \leq X} \frac{\tau(m)^{k} \tau(n)^{k}}{[m,n]} &\leq& \sum_{m,n \leq X} \frac{\tau(m)^{k} \tau(n)^{k}}{mn} \sum_{d|m,n} d, \\
&\leq& \sum_{d \leq X} \frac{1}{d} \tau(d)^{2k} \sum_{m',n' \leq X} \frac{\tau(m')^{k} \tau(n')^{k}}{m'n'}.
\end{eqnarray*}
This, together with the second estimate, gives the third. 
The last estimate follows from Theorem 1 in \cite{Shiu1980ABT}; alternatively, it follows from Landreau's inequality.
\end{proof}

This leads to the estimate
\begin{eqnarray} \label{LR bound}
\Lambda_R(n) \ll
\begin{cases}
R & \text{if } n = 0, \\
X^{o(1)} & \text{if } 0<|n| \ll X,
\end{cases}
\end{eqnarray}
which we will use several times in the following sections. 

\begin{lem}[Bombieri-Vinogradov theorem]
For $d \geq 1$, let $E(X,d)$ denote
$$ E(X,d) = \max_{l \in \mathbb{Z}, t \leq X} \left| \sum_{\substack{n \leq t \\ n \equiv l \bmod d}} \Lambda(n) - \frac{t}{\varphi(d)} \1_{(d,l)=1} \right|. $$
For any $A,k>0$, we have
$$ \sum_{d \leq X^{(1-\varepsilon)/2}} \tau(d)^k E(X,d) \ll \frac{X}{(\ln{X})^A}. $$
\end{lem}

Let $f$ be a $1$-bounded arithmetic function and $l:\mathbb{N} \to \mathbb{Z}$ be a function which satisfies 
$$ \sum_{d \leq X} \frac{(d,l(d))}{d} \ll X^{o(1)}. $$

\begin{lem}\label{LR mod}
Let $R \geq 2, d \in \mathbb{Z}, l \in \mathbb{Z}, (d,l) \leq R$ and $A>0$. We have
$$ \sum_{\substack{n \leq t \\ n \equiv l \bmod d}} \Lambda_R(n) = \frac{t}{\varphi(d)} \1_{(d,l)=1} + O\left( \frac{X\tau(d)^2}{d(\ln{2R/(d,l)})^A} + R \ln{R} \right), $$
uniformly for all $t \leq X$. 
\end{lem}
\begin{proof}
The case with $t=X$ and $R \leq X$ is proved in Lemma 3.12. of \cite{tsuda2024small}. However, the assumption $R \leq X$ is not actually used in that argument. Thus we can drop this assumption and extend the result to all $t \leq X$. 
\end{proof}

\begin{lem}\label{ML-LR}[Major arc estimate]
Let $A>0, q \leq (\ln{X})^A, (q,a)=1, \alpha = \frac{a}{q} + \eta$ and $D = X^{(1- \varepsilon)/2}, R=X^{\beta}, \beta \in (0,1/2]$. We have
$$ \sum_{d \leq D} f(d) \sum_{\substack{n \leq X \\ n \equiv l(d) \bmod d}} (\Lambda(n) - \Lambda_R(n))e(n\alpha) \lep (1+X|\eta|) \frac{X}{(\ln{X})^B}, $$
for any $B>0$. 
\end{lem}
\begin{proof}
We split the sum over $n$ according to residue classes modulo $q$. This leads to
$$ = \sum_{b \in \res{q}} e(ab/q) \sum_{d \leq D} f(d) \sum_{\substack{n \leq X \\ n \equiv l(d) \bmod d \\ n \equiv b \bmod q}} (\Lambda(n) - \Lambda_R(n)) e(n\eta). $$
Thus, by partial summation, it suffices to show that
\begin{equation} \label{ML-LR eq1} \sum_{d \leq D} f(d) \sum_{\substack{n \leq t \\ n \equiv l(d) \bmod d \\ n \equiv b \bmod q}} (\Lambda(n) - \Lambda_R(n)) \ll \frac{X}{(\ln{X})^B} \end{equation}
for any $t \leq X, q \leq (\ln{X})^A, (b,q)=1$ and $B>0$.

We estimate the innermost sum with the condition $l(d) \equiv b \bmod (d,q)$; otherwise the sum is empty. By the Chinese remainder theorem, we can rewrite the condition as $n \equiv l' \bmod d'$ where $d'=[d,q]$ and $l'$ satisfies $(d',l') \leq (d,l') (q,l') = (d,l(d))$. Thus, if $(d,l(d)) \leq R^{1/2}$, we see that
$$ \sum_{\substack{n \leq t \\ n \equiv l(d) \bmod d \\ n \equiv b \bmod q}} (\Lambda(n) - \Lambda_R(n)) \lep E(X,[d,q]) + \frac{\tau([d,q])^2}{[d,q](\ln{X})^B}X + R, $$
by Lemma \ref{LR mod}. If $(d,l(d))>R^{1/2}$, we have
$$ \sum_{\substack{n \leq X \\ n \equiv l(d) \bmod d \\ n \equiv b \bmod q}} (\Lambda(n) - \Lambda_R(n)) \ll \frac{X^{1+o(1)}}{d} \leq \frac{X^{1+o(1)}}{d} \times \frac{(d,l(d))}{R^{1/2}}. $$

Therefore, the left-hand side of \eqref{ML-LR eq1} is bounded by
\begin{eqnarray*}
&& \sum_{d \leq D} \left( E(X,[d,q]) + \frac{\tau([d,q])^2}{[d,q](\ln{X})^B} X + R + \frac{X^{1+o(1)}}{d} \frac{(d,l(d))}{R^{1/2}} \right), \\
&\leq& \sum_{d' \leq qD} \left( \tau(d') E(X,d') + \frac{\tau(d')^3}{d' (\ln{X})^B} X \right) + DR + X^{1+o(1)} R^{-1/2} \sum_{d \leq X} \frac{(d,l(d))}{d}, \\
&\lep& \frac{X}{(\ln{X})^B}.
\end{eqnarray*}
\end{proof}

\begin{lem}\label{mL}[Minor arc estimate of $\Lambda$]
There exists a function $B(A) \to \infty$ as $A \to \infty$ such that the following holds.  

Let $A \geq 2, (\ln{X})^A \leq q \leq X (\ln{X})^{-A}, (q,a)=1, |\alpha - \frac{a}{q}| \leq q^{-2}$ and $D \leq X^{(1-\varepsilon)/3}$. Then, we have 
$$ \sum_{d \leq D} f(d) \sum_{\substack{n \leq X \\ n \equiv l \bmod d}} \Lambda(n) e(n\alpha) \lep \frac{X}{(\ln{X})^{B(A)}} $$

\end{lem}
\begin{proof}
This easily follows from Theorem in \cite{balog1985exponential}, or see Lemma 2 in \cite{balog1990prime}. 
\end{proof}

\begin{lem}\label{mLR}[Minor arc estimate for $\Lambda_R$]
Let $(q,a)=1, |\alpha - \frac{a}{q}| \leq q^{-2}$ and $1 \leq D,R \leq X$. Then, we have 
$$ \sum_{d \leq D} f(d) \sum_{\substack{n \leq X \\ n \equiv l \bmod d}} \Lambda_R(n) e(n\alpha) \lep Xq^{-1/2} + (XDR)^{1/2} + (qX)^{1/2} $$
\end{lem}
\begin{proof}
We expand $\Lambda_R$. This gives
$$ \sum_{d \leq D} f(d) \sum_{\substack{n \leq N \\ n \equiv l(d) \bmod d}} \Lambda_R(n) e(n\alpha) = \sum_{d \leq D} f(d) \sum_{r \leq R} \mu(r) \ln{R/r} \sum_{\substack{n\leq N \\ n \equiv l(d) \bmod d \\ n \equiv 0 \bmod r}} e(n\alpha). $$
Using a well-known estimate for exponential sums, we obtain
\begin{eqnarray*}
&\lep& \sum_{r \leq R} \sum_{d \leq D} \sup_{b} \left| \sum_{\substack{n\leq X \\ n \equiv b \bmod [d,r]}} e(n\alpha) \right| \\
&\ll& \sum_{s \leq DR} \tau(s) \sup_{b} \left| \sum_{\substack{n\leq X \\ n \equiv b \bmod s}} e(n\alpha) \right| \\
&\ll& \left( \sum_{s \leq DR} \tau(s)^{2} \frac{X}{s} \right)^{1/2} \left( \sum_{s \leq DR} \min\left( \frac{X}{s}, \frac{1}{\|s\alpha\|} \right) \right)^{1/2} \\
&\lep& X^{1/2} (Xq^{-1} + DR + q)^{1/2} \\
&\ll& Xq^{-1/2} + (XDR)^{1/2} + (qX)^{1/2}. 
\end{eqnarray*}
\end{proof}

\subsection{Proof of Theorem \ref{LL level}}

We split the sum over $d_1,d_2$ as follows:
$$ \sum_{(d_1,d_2) \in X^{(1-\varepsilon) \cdot \mathcal{D}}} = \sum_{(d_1,d_2) \in X^{(1-\varepsilon) \cdot [0,1/2]\times [0,1/3]}} + \sum_{(d_1,d_2) \in X^{(1-\varepsilon) \cdot [0,1/3]\times [0,1/2]}} - \sum_{(d_1,d_2) \in X^{(1-\varepsilon) \cdot [0,1/3]\times [0,1/3]}} $$
It suffices to consider the case $d_1 \leq D_1,d_2 \leq D_2$ with $D_1 = X^{(1 - \varepsilon)/3}, D_2 =X^{(1 - \varepsilon)/2}$, by rearranging $f_1,f_2$. 
Using orthogonality, we have 
\begin{equation} \label{01} \sum_{\substack{d_1 \leq D_1 \\ d_2 \leq D_2}}  f_1(d_1) f_2(d_2) \sum_{\substack{m+n=N \\ m \equiv l_1(d_1) \bmod d_1 \\ n \equiv l_2(d_2) \bmod d_2}} \Lambda(m) \Lambda(n) = \int_{\mathbb{T}} W_1(\alpha) W_2(\alpha) e(-N\alpha) d\alpha, \end{equation}
where
$$ W_i(\alpha) = \sum_{d \leq D_i} f_i(d) \sum_{\substack{n \leq X \\ n \equiv l_i(d) \bmod d}} \Lambda(n) e(n\alpha) $$
for $i=1,2$. Let
$$ E(N) = \sum_{d_1 \leq D_1} \sum_{d_2 \leq D_2} f_1(d_1) f_2(d_2) \sum_{\substack{m+n=N \\ m \equiv l_1 \bmod d_1 \\ n \equiv l_2 \bmod d_2}} (\Lambda(m) \Lambda(n) - \Lambda_{R}(m) \Lambda_{R}(n)). $$
Our theorem follows from the following estimate for any $B>0$:
\begin{equation} \label{sum EE} \sum_{X/2 < N \leq X} |E(N)|^2 \ll \frac{X^3}{(\ln{X})^{B}}. \end{equation}

We define $W^{\sharp}_i$ by
$$ W^{\sharp}_i(\alpha) = \sum_{d \leq D_i} f_i(d) \sum_{\substack{n \leq X \\ n \equiv l_i(d) \bmod d}} \Lambda_{R}(n) e(n\alpha) $$
for $i=1,2$. We now see that \eqref{sum EE} follows from
\begin{equation} \label{WW-W'W'} \int_{\mathbb{T}} |W_1(\alpha) W_2(\alpha) - W^{\sharp}_1(\alpha) W^{\sharp}_2(\alpha) |^2 d\alpha \ll \frac{X^3}{(\ln{X})^{B}}, \end{equation}
for any $B>0$, by Bessel's inequality. 

Let $A>0$ and
$$ P = (\ln{X})^A, \quad Q = X P^{-1}, $$
and we define the major arc
$$ \mathfrak{M} = \bigcup_{q \leq P} \bigcup_{a \in \res{q}} \mathfrak{M}_{q,a}, \quad \mathfrak{M}_{q,a} = \{ \alpha \in \mathbb{T} : \| \alpha - \frac{a}{q} \| \leq \frac{1}{qQ} \} $$
and the minor arc $\mathfrak{m} = \mathbb{T} \setminus \mathfrak{M}$.

Let $\alpha \in \mathfrak{M}$. Then, we can write $\alpha=\frac{a}{q}+ \eta$ with some $(q,a)=1, q \leq P$ and $|\eta| \leq \frac{P}{qX} \leq P X^{-1}$. By Lemma \ref{ML-LR}, we have
\begin{eqnarray*}
W_i(\alpha) - W^{\sharp}_i(\alpha) &=& \sum_{d \leq D_i} f_i(d) \sum_{\substack{n \leq X \\ n \equiv l_i \bmod d}} (\Lambda(n) - \Lambda_{R}(n)) e(n\alpha) \ll (1+X|\eta|) \frac{X}{(\ln{X})^{2A}} \ll \frac{X}{(\ln{X})^{A}}.
\end{eqnarray*}
for $i=1,2$.

Let $\alpha \in \mathfrak{m}$. Then, there exists $(q,a) =1, P<q \leq Q$ such that $|\alpha - \frac{a}{q}| \leq \frac{1}{qQ} \leq \frac{1}{q^2}$, by Dirichlet's approximation theorem and the definition of minor arc. By Lemma \ref{mLR}, we have
$$ W^{\sharp}_1(\alpha) \ll X P^{-1/2} + (X D_1 X^{\beta})^{1/2} + (QX)^{1/2} \ll X P^{-1/2}. $$
and by Lemma \ref{mL}, we also have
$$ W_1(\alpha) \ll \frac{X}{(\ln{X})^{B(A)}} $$
for some $B(A)$ which satisfies that $B(A) \to \infty$ as $A \to \infty$.

By combining two estimates, we obtain
\begin{equation} \label{W-W' unif} \sup_{\alpha \in [0,1]} |W_1(\alpha) - W^{\sharp}_1(\alpha)| \ll \frac{X}{(\ln{X})^{B}} \end{equation}
for any $B>0$. 

We shall prove \eqref{WW-W'W'}. By
\begin{eqnarray*}
|W_1(\alpha) W_2(\alpha) - W^{\sharp}_1(\alpha) W^{\sharp}_2(\alpha)|^2 &=& |(W_1(\alpha) - W^{\sharp}_1(\alpha)) W_2(\alpha) + W^{\sharp}_1(\alpha) (W_2(\alpha) - W^{\sharp}_2(\alpha))|^2 \\
&\leq& 2 |(W_1(\alpha) - W^{\sharp}_1(\alpha)) W_2(\alpha)|^2 + 2 |W^{\sharp}_1(\alpha) (W_2(\alpha) - W^{\sharp}_2(\alpha))|^2,
\end{eqnarray*}
we see that \eqref{WW-W'W'} follows from
$$ \int_{\mathbb{T}} |W_1(\alpha) - W^{\sharp}_1(\alpha)|^2 |W_2(\alpha)|^2 d\alpha, \int_{\mathbb{T}} |W^{\sharp}_1(\alpha)|^2 |W_2(\alpha) - W^{\sharp}_2(\alpha)|^2 d\alpha \ll \frac{X}{(\ln{X})^{B}}, $$
for any $B>0$. 
The former term is bounded by
$$ \leq \sup_{\alpha' \in [0,1]} |W_1(\alpha') - W^{\sharp}_1(\alpha')|^2 \int_{\mathbb{T}} |W_2(\alpha)|^2 d\alpha \ll \frac{X^2}{(\ln{X})^{B}} \int_{\mathbb{T}} |W_2(\alpha)|^2 d\alpha $$
by \eqref{W-W' unif}. 

The latter term is
\begin{eqnarray*} 
&=& \left( \int_{\mathfrak{M}} + \int_{\mathfrak{m}} \right) |W^{\sharp}_1(\alpha)|^2 |W_2(\alpha) - W^{\sharp}_2(\alpha)|^2 d\alpha \\
&\ll& \sup_{\alpha' \in \mathfrak{M}} |W_2(\alpha') - W^{\sharp}_2(\alpha')|^2 \int_{\mathbb{T}} |W^{\sharp}_1(\alpha)|^2 d\alpha + \sup_{\alpha' \in \mathfrak{m}} |W^{\sharp}_1(\alpha')|^2 \int_{\mathbb{T}} (|W_2(\alpha)|^2 + |W^{\sharp}_2(\alpha)|^2) d\alpha \\
&\ll& \frac{X^2}{(\ln{X})^{B(A)}} \int_{\mathbb{T}} (|W^{\sharp}_1(\alpha)|^2 + |W_2(\alpha)|^2 + |W^{\sharp}_2(\alpha)|^2) d\alpha. 
\end{eqnarray*}
Here, $B(A) \to \infty$ as $A \to \infty$. 

The integrals of $|W^{\sharp}_1(\alpha)|^2, |W_2(\alpha)|^2,|W^{\sharp}_2(\alpha)|^2$ are handled similarly. For instance, the integral of $|W^{\sharp}_1(\alpha)|^2$ becomes \begin{eqnarray*} 
&=& \sum_{n \leq X} \Lambda_{R}(n)^2 \left| \sum_{\substack{d_1 \leq D_1 \\ n \equiv l_1(d_1) \bmod d_1}} f_1(d_1) \right|^2 \\
&\lep& \sum_{d_1 \leq D_1} \sum_{d'_1 \leq D_1} \sum_{\substack{n \leq X \\ n \equiv l_1(d_1) \bmod d_1 \\ n \equiv l_1(d'_1) \bmod d'_1}} \tau(n)^2 \\
&\lep& \sum_{d_1 \leq D_1} \sum_{d'_1 \leq D_1} \frac{\tau(d_1)^{2} \tau(d'_1)^{2}}{[d_1,d'_1]} \lep X. 
\end{eqnarray*}

\subsection{Extension}

We can extend Theorem \ref{LL level} to the following result. 
\begin{cor}\label{LL level ext}
Let $k$ be a positive integer. Let $\Omega_i(d) \subseteq \mathbb{Z}/d\mathbb{Z}, i=1,2$ that satisfy 
$$ \sum_{d \leq X} \frac{1}{d} \sum_{l \in \Omega_i(d)} (d,l) \ll X^{o(1)} $$
and $|\Omega_i(d)| \ll \tau(d)^k$.  
Let $f_1,f_2:\mathbb{N} \times \mathbb{Z} \to \mathbb{C}$ be functions satisfying $f_1(l,d), f_2(l,d) \ll (\tau(d) \ln{X})^k$ for all $d \leq X, l \in \mathbb{Z}$. For any $B>0$, the following holds:

Let $R=X^{\beta},\beta \in (0,1/2]$. For all even integers $N \in (X/2,X]$ with $O(X (\ln{X})^{-B} )$ exceptions, we have
\begin{eqnarray} \label{level2} \sum_{(d_1,d_2) \in X^{(1-\varepsilon)\mathcal{D}}} \sum_{\substack{l_1 \in \Omega_1(d_1) \\ l_2 \in \Omega_2(d_2)}} f_1(l_1,d_1) f_2(l_2,d_2) E_{R}(l_1,l_2,d_1,d_2;N) \ll \frac{X}{(\ln{X})^B}.
\end{eqnarray}
\end{cor}

\begin{proof}
Fix $A>0$. We first consider the case when $\Omega_i, f_i, i=1,2$ satisfy $|\Omega_i(d)| \leq (\ln{X})^{A}, f_i(l,d) \leq (\ln{X})^{A}$ for all $d \leq X$ and $l$. In this case, we can express \eqref{level2} as
$$ \sum_{j=1}^{J} \sum_{(d_1,d_2) \in X^{(1-\varepsilon) \cdot \mathcal{D}}} f^{(j)}_1(d_1) f^{(j)}_2(d_2) E_{R}(l^{(j)}_1(d_1),l^{(j)}_2(d_2),d_1,d_2;N) $$
for some functions $f^{(j)}_i,l^{(j)}_i, i=1,2, 1 \leq j \leq J = (\ln{X})^{2A}$ satisfying $f^{(j)}_i(d) \leq (\ln{X})^A$ and 
$$ \sum_{d \leq X} \frac{(d,l^{(j)}_i(d))}{d} \ll X^{o(1)}, $$
By Theorem \ref{LL level}, for each $j$, we see that
$$ \sum_{(d_1,d_2) \in X^{(1-\varepsilon) \cdot \mathcal{D}}} \frac{f^{(j)}_1(d_1)}{(\ln{X})^A} \frac{f^{(j)}_2(d_2)}{{(\ln{X})^A}} E_{R}(l^{(j)}_1(d_1),l^{(j)}_2(d_2),d_1,d_2;N) \ll \frac{X}{(\ln{X})^{B}} $$
all $N \in (X/2,X]$ but $O(X (\ln{X})^{-B})$ exceptions, for any $B>0$. Therefore, we have
$$ \sum_{(d_1,d_2) \in X^{(1-\varepsilon) \cdot \mathcal{D}}} \sum_{\substack{l_1 \in \Omega_1(d_1) \\ l_2 \in \Omega_2(d_2)}} f_1(l_1,d_1) f_2(l_2,d_2) E_{R}(l_1,l_2,d_1,d_2;N) \ll \frac{X}{(\ln{X})^{B-4A}} $$
for all $N \in (X/2,X]$ but for $O(X (\ln{X})^{-B+2A})$ exceptions. The claim is now proved. 

We now consider the general case. Define $f'_i(l,d) = \1_{\substack{\tau(d) \leq (\ln{X})^{A/k-1} \\ }} f_i(l,d)$ and split the sum in \eqref{level2} depending on whether $\tau(d) \leq (\ln{X})^{A/k-1}$ or not. This gives
\begin{multline}
\label{ext 1 div}
 = \sum_{(d_1,d_2) \in X^{(1-\varepsilon) \cdot \mathcal{D}}} \sum_{\substack{l_1 \in \Omega_1(d_1) \\ l_2 \in \Omega_2(d_2)}} f'_1(l_1,d_1) f'_2(l_2,d_2) E_{R}(l_1,l_2,d_1,d_2;N) \\
 + O\left( (\ln{X})^{2k} \sum_{\substack{d_1,d_2 \leq X \\ \tau(d_1) > (\ln{X})^{A/k-1} \text{ or }  \tau(d_2) > (\ln{X})^{A/k-1}}} \sum_{\substack{l_1 \in \Omega_1(d_1) \\ l_2 \in \Omega_2(d_2)}} \tau(d_1)^{k} \tau(d_2)^{k} |E_{R}(l_1,l_2,d_1,d_2;N)| \right)
\end{multline}
The first term is admissible by previous argument since $\tau(d) \leq (\ln{X})^{A/k-1}$ implies that $|\Omega_i(d)| \leq (\ln{X})^{A}, f_i(l,d) \leq (\ln{X})^{A}$. We sum over $N$ for the second term. We now see that for any $d_1,d_2 \leq X$, 
\begin{eqnarray*}
\sum_{X/2 < N \leq X} |E_R(l_1,l_2,d_1,d_2;N)| &\lep& \sum_{N \leq X} \sum_{\substack{m+n=N \\ m \equiv l_1 \bmod d_1 \\ n \equiv l_2 \bmod d_2}} \tau(m) \tau(n) \\
&\leq& \left( \sum_{\substack{m \leq X \\ m \equiv l_1 \bmod d_1}} \tau(m) \right) \left( \sum_{\substack{n \leq X \\ n \equiv l_2 \bmod d_2}} \tau(n) \right) \\
&\lep& X^2 \frac{\tau(d_1) \tau(d_2)}{d_1 d_2},
\end{eqnarray*}
so that the entire sum becomes
\begin{eqnarray*}
&\lep& X^2 (\ln{X})^{2k} \sum_{\substack{d_1,d_2 \leq X \\ \tau(d_1) > (\ln{X})^{A/k-1} \text{ or } \tau(d_2) > (\ln{X})^{A/k-1}}} \frac{\tau(d_1)^{k+1} \tau(d_2)^{k+1}}{d_1 d_2} \\
&\leq& X^2 (\ln{X})^{(1-A/k) + 2k} \sum_{d_1,d_2 \leq X} \frac{\tau(d_1)^{k+2} \tau(d_2)^{k+2}}{d_1 d_2} \ll X^2 (\ln{X})^{-A/k + O_k(1)} \\
\end{eqnarray*}
Thus the $O$-term in \eqref{ext 1 div} is admissible for almost all $N$, by taking $A$ large enough. 
\end{proof}

\subsection{Level of distribution of $\{(m,n) : m+n=N\}$}

\begin{thm}[Variants of level of distribution of $\{(m,n) : m+n=N\}$]\label{11 level}
Let $l_1,l_2:\mathbb{N} \to \mathbb{Z}$ and $\varepsilon>0, D=X^{1-\varepsilon}$. Suppose
\begin{equation} \label{l1l2} \sum_{D<n \leq X} \left( \sum_{\substack{d \leq D \\ n \equiv l_i(d) \bmod d}} 1\right)^2 \ll X^{1+o(1)} \end{equation}
holds for at least one of $i=1,2$. Let $f_1,f_2$ be $1$-bounded arithmetic functions. 

For all integers $N \in (X/2,X]$ with $O(X^{1-\varepsilon/8} )$ exceptions, we have
\begin{equation} \label{level 1} \sum_{d_1,d_2 \leq D} f_1(d_1) f_2(d_2) \left( \sum_{\substack{m+n=N \\ m \equiv l_1(d_1) \bmod d_1 \\ n \equiv l_2(d_2) \bmod d_2}} 1 - \frac{N}{[d_1,d_2]} \1_{l_1(d_1) + l_2(d_2) \equiv N \bmod (d_1,d_2)} \right) \ll X^{1-\varepsilon/8}.
\end{equation}
\end{thm}
\begin{proof}

Suppose \eqref{l1l2} holds with $i=1$. The term with $m \leq D$ contributes
$$ \leq \sum_{d_1,d_2 \leq D} \sum_{\substack{m+n=N \\ m \equiv l_1(d_1) \bmod d_1 \\ n \equiv l_2(d_2) \bmod d_2 \\ m \leq D}} 1. $$
to \eqref{level 1}. Summing over $N \in (X/2,X]$, this becomes
$$ \leq \sum_{d_1,d_2 \leq D} \sum_{\substack{m+n \leq X \\ m \equiv l_1(d_1) \bmod d_1 \\ n \equiv l_2(d_2) \bmod d_2 \\ m \leq D}} 1 \leq \left( \sum_{d_1 \leq D} \sum_{\substack{\substack{m \leq D \\ m  \equiv l_1(d_1) \bmod d_1}}} 1\right) \left( \sum_{d_2 \leq D} \sum_{\substack{\substack{n \leq X \\ n  \equiv l_2(d_2) \bmod d_2}}} 1\right) \lep XD. $$
Similarly for the term with $n \leq D$.

Let 
$$ P = X^{2\varepsilon/7}, \quad Q = X P^{-1} $$
and define the major arc
$$ \mathfrak{M} = \bigcup_{q \leq P} \bigcup_{a \in \res{q}} \mathfrak{M}_{q,a}, \quad \mathfrak{M}_{q,a} = \{ \alpha \in \mathbb{T} : \| \alpha - \frac{a}{q} \| \leq \frac{1}{qQ} \} $$
and the minor arc $\mathfrak{m} = \mathbb{T} \setminus \mathfrak{M}$. We now have
\begin{eqnarray*}
\sum_{d_1,d_2 \leq D} f_1(d_1) f_2(d_2) \sum_{\substack{m+n=N \\ m \equiv l_1 \bmod d_1 \\ n \equiv l_2 \bmod d_2 \\ m,n >D}} 1 &=& \left( \int_{\mathfrak{M}} + \int_{\mathfrak{m}} \right) W_1(\alpha) W_2(\alpha) e(-N\alpha) d\alpha \\
&\eqcolon& I_{\mathfrak{M}}(N) + I_{\mathfrak{m}}(N), 
\end{eqnarray*}
where
$$ W_i(\alpha) = \sum_{d \leq D} f_i(d) \sum_{\substack{D<n \leq X \\ n \equiv l_i(d) \bmod d}} e(n\alpha), $$
for $i=1,2$. Note that \eqref{l1l2} implies
$$ \int_{\mathbb{T}} |W_1(\alpha)|^2 d\alpha =  \sum_{D<n \leq X} \left| \sum_{\substack{d \leq D \\ n \equiv l_1(d) \bmod d}} f_1(d) \right|^2 \ll X^{1+o(1)}. $$
Let $P < q \leq Q, (q,a)=1, |\alpha - \frac{a}{q}| \leq q^{-2}$. Then, we have 
\begin{eqnarray*}
|W_i(\alpha)| &\leq& \sum_{d \leq D} \sup_{b} \left| \sum_{\substack{D<n \leq X \\ n \equiv b \bmod d}} e(n\alpha) \right| \\
&\leq& \sum_{d \leq D} \min\left( \frac{X}{d}, \frac{1}{\|d\alpha\|} \right) \\
&\lep& \frac{X}{q} + D + q \ll XP^{-1}
\end{eqnarray*}
for $i=1,2$. Now let $q \leq X, (q,a)=1$ and $\alpha = \frac{a}{q} + \eta$. Then, by partial summation, we have 
\begin{eqnarray*}
W_i(\alpha) &=& \sum_{d \leq D} f_i(d) \sum_{b=1}^{q} e(ab/q) \sum_{\substack{D<n \leq X \\ n \equiv l_i(d) \bmod d \\ n \equiv b \bmod q}} e(n\eta) \\
&=& \sum_{d \leq D} f_i(d) \sum_{\substack{b=1 \\ b \equiv l_i(d) \bmod (d,q)}}^{q} e(ab/q) \sum_{\substack{D<n \leq X \\ n \equiv l_i(d) \bmod d \\ n \equiv b \bmod q}} e(n\eta) \\
&=& \sum_{d \leq D} f_i(d) \sum_{\substack{b=1 \\ b \equiv l_i(d) \bmod (d,q)}}^{q} e(ab/q) \left( \frac{1}{[d,q]} \int_{D}^{X} e(\eta t) dt + O(1+X|\eta|) \right) \\
&=& v(\eta) \sum_{d \leq D} f_i(d) \frac{1}{[d,q]} \sum_{\substack{b=1 \\ b \equiv l_i(d) \bmod (d,q)}}^{q} e(ab/q) + O(qD(1+X|\eta|))
\end{eqnarray*}
where $v(\eta)=\int_{D}^{X} e(\eta t) dt$. We also note that
$$ |W_i(\alpha)| \leq \sum_{d \leq D} \sum_{\substack{D<n \leq X \\ n \equiv l_i(d) \bmod d}} 1 \ll X (\ln{X}) $$
holds for all $\alpha \in \mathbb{T}$. 

Thus we have
$$ \sum_{X/2 < N \leq X} |I_{\mathfrak{m}}(N)|^2 \leq \int_{\mathfrak{m}} |W_1(\alpha) W_2(\alpha)|^2 d\alpha \leq (\sup_{\alpha \in \mathfrak{m}} |W_2(\alpha)|^2) \int_{\mathbb{T}} |W_1(\alpha)|^2 d\alpha \lep X^{3+o(1)} P^{-2} $$
and
\begin{align}
I_{\mathfrak{M}}(N) &= \sum_{q \leq P} \sum_{a \in \res{q}} e(-aN/q) \int_{-1/qQ}^{1/qQ} W_1(a/q+\eta) W_2(a/q+\eta) e(-N\eta) d\eta \notag \\
&= \sum_{d_1,d_2 \leq D} f_1(d_1) f_2(d_2) \sum_{q \leq P} \frac{1}{[d_1,q] [d_2,q]}\sum_{a \in \res{q}} e(-aN/q) \label{IM maint} \\
& \times \sum_{\substack{b_1=1 \\ b_1 \equiv l_1(d_1) \bmod (d_1,q)}}^{q} e(ab_1/q) \sum_{\substack{b_2=1 \\ b_2 \equiv l_2(d_2) \bmod (d_2,q)}}^{q} e(ab_2/q) \int_{-1/qQ}^{1/qQ} v(\eta)^2 e(-N\eta) d\eta \notag \\ 
& + O\left( (\ln{X}) \sum_{q \leq P} \sum_{a \in \res{q}} \int_{-1/qQ}^{1/qQ} XDP d\eta \right). \notag
\end{align}
The error term here is bounded by $\lep D P^3$. 

The integral becomes
$$ \int_{-1/qQ}^{1/qQ} v(\eta)^2 e(-N\eta) d\eta = \int_{-\infty}^{\infty} v(\eta)^2 e(-N\eta) d\eta + O\left( \int_{1/qQ}^{\infty} |v(\eta)|^2 d\eta \right) =N+O(qQ), $$
using the Fourier inversion formula, $|v(\eta)| \ll |\eta|^{-1}$ and $D \ll qQ$. We temporarily fix $d_1,d_2$ and consider the sum over $q$. We abbreviate $l_1(d_1),l_2(d_2)$ as $l_1,l_2$. 
Rewriting $b=l+(d,q) b'$ with $b' \in \mathbb{Z}/(q/(d,q))\mathbb{Z}$, we see that
$$ \sum_{\substack{b=1 \\ b \equiv l \bmod (d,q)}}^{q} e(ab/q) = e(al/q) \sum_{b'=1}^{q/(d,q)} e\left(\frac{ab'}{q/(d,q)}\right) = \frac{q}{(d,q)} e(al/q) \1_{q/(d,q)|a}. $$
Thus, the sum over $q$ in \eqref{IM maint} becomes
\begin{eqnarray*}
&& \sum_{q \leq P} \frac{1}{d_1 d_2} (N + O(qQ)) \times \sum_{\substack{a \in \res{q} \\ q/(d_1,q)|a \\ q/(d_1,q)|a}} e(a(l_1+l_2-N)/q) \\
&=& \sum_{q \leq P} \frac{1}{d_1 d_2} (N + O(qQ)) c_q(l_1+l_2-N) \1_{q/(d_1,q)=q/(d_2,q)=1} \\
&=& \frac{1}{d_1 d_2} \left( N \left( \sum_{q|(d_1,d_2)} - \sum_{\substack{q > P \\ q|(d_1,d_2)}} \right) c_q(l_1+l_2-N) + O\left(qQ \sum_{\substack{q \leq P \\ q|(d_1,d_2)}} |c_q(l_1+l_2-N)| \right) \right) \\
&=& \frac{N}{d_1 d_2} \sum_{a=1}^{(d_1,d_2)} e(a(l_1+l_2-N)/(d_1,d_2)) + O\left( \frac{1}{d_1 d_2} \sum_{\substack{q|(d_1,d_2)}} |c_q(l_1+l_2-N)| (N \1_{P<q}+qQ) \right) \\
&=& \frac{1}{[d_1,d_2]} \1_{l_1 + l_2 \equiv N \bmod (d_1,d_2)} + O\left( \frac{X}{d_1 d_2} \sum_{\substack{q|(d_1,d_2)}} (l_1+l_2-N,q) (\1_{P<q}+qP^{-1}) \right).
\end{eqnarray*}

We now sum over $N$ and $d_1,d_2$ for the error term. This becomes
\begin{eqnarray*}
&& X \sum_{d_1,d_2 \leq D} \frac{1}{d_1 d_2} \sum_{\substack{q|(d_1,d_2)}} (\1_{P<q}+qP^{-1}) \sum_{X/2<N\leq X} (l_1(d_1)+l_2(d_2)-N,q) \\
&\ll& X^2 \sum_{d_1,d_2 \leq D} \frac{1}{d_1 d_2} \sum_{\substack{q|(d_1,d_2)}} \tau(q) (\1_{P<q}+qP^{-1}) \\
&\leq& X^2 \left( \sum_{q>P} \tau(q) \sum_{\substack{d_1,d_2 \leq D \\ q|d_1,d_2}} \frac{1}{d_1 d_2} + P^{-1} \sum_{d_1,d_2 \leq D} \frac{(d_1,d_2)}{d_1 d_2} \tau((d_1,d_2))^2 \right) \lep X^2 P^{-1} \\
\end{eqnarray*}

Putting everything together, we have
\begin{equation} \label{<E1E2E3} \sum_{d_1,d_2 \leq D} f_1(d_1) f_2(d_2) \left( \sum_{\substack{m+n=N \\ m \equiv l_1 \bmod d_1 \\ n \equiv l_2 \bmod d_2}} 1 - \frac{N}{[d_1,d_2]} \1_{l_1(d_1) + l_2(d_2) \equiv N \bmod (d_1,d_2)} \right) = E_1(N) + E_2(N) + E_3(N) \end{equation}
with some function $E_1,E_2,E_3$ satisfying
$$ E_1(N) \lep D P^3, \quad \sum_{X/2<N\leq X} |E_2(N)| \lep X^2 P^{-1}, \quad \sum_{X/2<N\leq X} |E_3(N)|^2 \lep X^{3+o(1)} P^{-2}. $$
From this, we see that all $N \in (X/2,X]$ but for $\ll XP^{-1/2}$ exceptions, \eqref{<E1E2E3} is bounded by $\lep XP^{-1/2}$. This gives \eqref{level 1}. 

\end{proof}

\begin{cor}\label{11 level ext}
Let $\varepsilon>0, D=X^{1-\varepsilon}, k \in \mathbb{N}$. Let $\Omega_i(d) \subseteq \mathbb{Z}/d\mathbb{Z}, i=1,2$ that satisfy $|\Omega_i(d)| \ll X^{o(1)}, i=1,2$ and 
\begin{equation} \label{l1l2 c} \sum_{D<n\leq X} \left( \sum_{\substack{d \leq D \\ n \pmod d \in \Omega_i(d)}} 1 \right)^2 \ll X^{1+o(1)} \end{equation}
for at least one of $i=1,2$. Let $f_1,f_2$ be arithmetic functions satisfying $f_1(l,d), f_2(l,d) \ll X^{o(1)}$ for all $d \leq X, l \in \mathbb{Z}/d\mathbb{Z}$. 
For any $B>0$ and $Y\geq X^2$, the following holds:

For all integers $N \in (X/2,X]$ with $O(X^{1-\varepsilon/10})$ exceptions, we have
\begin{equation} \label{l1_c} \sum_{d_1,d_2 \leq D} \sum_{\substack{l_1 \in \Omega_1(d_1) \\ l_2 \in \Omega_2(d_2)}} f_1(l_1,d_1) f_2(l_2,d_2) \left( \sum_{\substack{m+n=N \\ m \equiv l_1 \bmod d_1 \\ n \equiv l_2 \bmod d_2}} 1 - \frac{N}{Y} \sum_{\substack{m \leq Y \\ m \equiv l_1 \bmod d_1 \\ N-m \equiv l_2 \bmod d_2}} 1 \right) \ll X^{1-\varepsilon/10}
\end{equation}

\end{cor}

\begin{proof}
Let $l_1,l_2:\mathbb{N} \to \mathbb{Z}$ and $f_1,f_2$ be $1$-bounded arithmetic functions satisfying $l_i(d) \in \Omega_i(d)$ for all $d \leq D$. Note that \eqref{l1l2 c} implies
$$ \sum_{D<n\leq X} \left( \sum_{\substack{d \leq D \\ n \equiv l_i(d) \bmod d}} 1 \right)^2 \leq \sum_{D<n\leq X} \left( \sum_{\substack{d \leq D \\ n \pmod d \in \Omega_i(d)}} 1 \right)^2 \ll X^{1+o(1)} $$
By the Chinese remainder theorem, we have
$$ \sum_{\substack{m \leq Y \\ m \equiv l_1(d_1) \bmod d_1 \\ N-m \equiv l_2(d_2) \bmod d_2}} 1 = \frac{Y}{[d_1,d_2]} \1_{l_1(d_1) + l_2(d_2) \equiv N \bmod (d_1,d_2)} + O(1) $$
for any $d_1,d_2,N$. This gives
\begin{align*}
&\sum_{d_1,d_2 \leq D} f_1(d_1) f_2(d_2) \left( \frac{N}{Y} \sum_{\substack{m \leq Y \\ m \equiv l_1(d_1) \bmod d_1 \\ N-m \equiv l_2(d_2) \bmod d_2}} 1 - \frac{N}{[d_1,d_2]} \1_{l_1(d_1) + l_2(d_2) \equiv N \bmod (d_1,d_2)} \right) \\
& \ll XD^2 Y^{-1} \leq X^{1-2\varepsilon}.  
\end{align*}
for all $N \in (X/2,X]$. Subtracting this from \eqref{level 1}, we have
\begin{equation} \label{l1_dec} \sum_{d_1,d_2 \leq D} f_1(d_1) f_2(d_2) E(l_1(d_1),l_2(d_2),d_1,d_2;N) \ll X^{1-\varepsilon/8} \end{equation}
for almost all $N$, by Theorem \ref{11 level}. Here, 
$$ E(l_1,l_2,d_1,d_2;N) = \sum_{\substack{m+n=N \\ m \equiv l_1 \bmod d_1 \\ n \equiv l_2 \bmod d_2}} 1 - \frac{N}{Y} \sum_{\substack{m \leq Y \\ m \equiv l_1 \bmod d_1 \\ N-m \equiv l_2 \bmod d_2}} 1. $$

Since the left-hand side of \eqref{l1_c} can be written as a linear combination of $\ll X^{o(1)}$ sums of the form \eqref{l1_dec}, the claim follows.

\end{proof}

\section{Main Lemmas}\label{ML}

\subsection{Correlation of $\Lambda_Q$}

We present an extension of an asymptotic formula of the following type that is valid in specific settings. 
\begin{lem}\label{GYlem}
Let $(\alpha_i)_{i=1}^{k},(\alpha'_i)_{i=1}^{k} \in (0,1)^k, Q_i=X^{\alpha_i}, Q'_i = X^{\alpha'_i}$ and $(\varepsilon_i)_{i=1}^{k} \in \{0,1\}^k$. Let $\mathcal{H}=\{h_1,\ldots,h_k\}$ be a set of $k$ distinct integers satisfying $|h_i| \leq X^2$ for each $i$. We have
\begin{equation} \label{GY cor} \sum_{n \leq X} \prod_{i=1}^{k} \Lambda_{Q_i}(n+h_i) \Lambda_{Q'_i}(n+h_i)^{\varepsilon_i} = X (\mathfrak{S}(\mathcal{H}) \prod_{i=1}^{k} (\min(\ln{Q_i},\ln{Q'_i}))^{\varepsilon_i}  + o(1)) + O\left(\prod_{i=1}^{k} Q_i Q'^{\varepsilon_i}_i\right), \end{equation}
where $\mathfrak{S}(\mathcal{H})$ denotes
$$ \mathfrak{S}(\mathcal{H}) = \prod_{p} \left(1 - \frac{|\mathcal{H}/p\mathbb{Z}|}{p}\right) \left(1 - \frac{1}{p}\right)^{-k}, $$
and the $o$-term depends only on $k$ and $(\alpha_i)_{i=1}^{k},(\alpha'_i)_{i=1}^{k}$ (uniform with respect to $\mathcal{H}$).  
\end{lem}
\begin{proof}
This follows as a special case of Theorem 8.1 in \cite{goldston2007higher3}. 
\end{proof}
Heuristically, this can be understood as
$$ \Lambda_{Q}(n) \Lambda_{Q'}(n)^{\varepsilon} \approx \Lambda(n) (\min(\ln{Q},\ln{Q'}))^{\varepsilon}, \quad \varepsilon \in \{0,1\} $$
as suggested by a simple calculation using Lemma \ref{LR mod}, at least when $\alpha \neq \alpha'$.

\begin{comment}
Let $\alpha,\alpha'>0$ and $Q=X^{\alpha}, Q' = X^{\alpha'}$. By Lemma \ref{LR mod} ($d=l=1$), we see that
$$ \sum_{n \leq X} \Lambda_{Q}(n) = X(1+o(1)) + O(Q \ln{Q}) = (1+o(1)) \sum_{n \leq X} \Lambda(n), $$
if $\alpha <1$. We also have
\begin{eqnarray*}
\sum_{n \leq X} \Lambda_{Q}(n) \Lambda_{Q'}(n) &=& \sum_{d \leq Q} \mu(d) \ln{Q/d} \sum_{\substack{n \leq X \\ n \equiv 0 \bmod d}} \Lambda_{Q'}(n) \\
&=& \sum_{d \leq Q} \mu(d) \ln{Q/d} \left( \frac{X}{\varphi(d)} \1_{d=1} + O\left( \frac{X\tau(d)^2}{d(\ln{2Q'/d})^A} + Q' \ln{Q'} \right) \right) \\
&=& X \ln{Q} + O\left(\frac{X (\ln{Q})^4 }{(\ln{Q'/Q})^A}\right) + O(QQ' (\ln{Q}) (\ln{Q'})) \\
&=& (1+o(1)) \sum_{n \leq X} \Lambda(n) (\min(\alpha_i,\alpha'_i) \ln{X}) 
\end{eqnarray*}
by Lemma \ref{LR mod} if $\alpha<\alpha'$ and $\alpha+\alpha'<1$. Thus we may expect that 
\begin{equation} \Lambda_{Q}(n) \Lambda_{Q'}(n)^{\varepsilon} \approx \Lambda(n) (\min(\alpha,\alpha') \ln{X})^{\varepsilon} \end{equation}
for $\varepsilon \in \{0,1\}$. From this, we expect that
$$ \sum_{n \leq X} \prod_{i=1}^{k} \Lambda_{Q_i}(n+h_i) \Lambda_{Q'_i}(n+h_i)^{\varepsilon_i} \approx \prod_{i=1}^{k} (\min(\alpha_i,\alpha'_i) \ln{X})^{\varepsilon_i} \times \sum_{n \leq X} \prod_{i=1}^{k} \Lambda(n+h_i). $$
The Hardy-Littlewood conjecture imply the left-hand side here is equal to the left-hand side of \eqref{GY cor}. 
\end{comment}

\begin{lem}[Main lemma]\label{mainlem}
Let $(\alpha_i)_{i=1}^{k},(\beta_i)_{i=1}^{k},(\alpha'_i)_{i=1}^{k},(\beta'_i)_{i=1}^{k} \in (0,1)^k$ and $Q_i = X^{\alpha_i}, Q'_i=X^{\alpha'_i},R_i = X^{\beta_i}, R'_i=X^{\beta'_i}, Q_i \leq Q'_i, R_i \leq R'_i$. Let $\mathcal{G}=\{g_1,\ldots,g_k\}, \mathcal{H}=\{h_1,\ldots,h_k\}$ be sets of $k$ distinct integers satisfying $|g_i|,|h_i| \leq (\ln{X})^3$ and $g,h$ be integers with $|g|,|h| \leq (\ln{X})^3$.   
Let
$$ S_0 = \sum_{m+n=N} \prod_{i=1}^{k} \Lambda_{Q_i}(m+g_i) \Lambda_{R_i}(n+h_i), $$
$$ S_1 = \sum_{m+n=N} \prod_{i=1}^{k} \Lambda_{Q_i}(m+g_i) \Lambda_{Q'_i}(m+g_i) \Lambda_{R_i}(n+h_i) \Lambda_{R'_i}(n+h_i), $$
$$ S_2 = \sum_{m+n=N} \Lambda(m+g) \Lambda(n+h) \prod_{i=1}^{k} \Lambda_{Q_i}(m+g_i) \Lambda_{Q'_i}(m+g_i) \Lambda_{R_i}(n+h_i) \Lambda_{R'_i}(n+h_i). $$
Let $A, \varepsilon>0$. The following statements hold for all integers $N \in (X/2,X]$ with $O(X(\ln{X})^{-A})$ exceptions. 

If
$$ \left( \sum_{i=1}^{k} \alpha_i, \sum_{i=1}^{k} \beta_i \right) \in (1-\varepsilon) \cdot [0,1] \times [0,1], $$
we have
\begin{equation} \label{S0goal} S_0 = N (\mathfrak{S}(\mathcal{G},\mathcal{H};N) + o(1)). \end{equation}
If
$$ \left( \sum_{i=1}^{k} (\alpha_i+\alpha'_i), \sum_{i=1}^{k} (\beta_i+\beta'_i) \right) \in (1-\varepsilon) \cdot [0,1] \times [0,1], $$
we have
\begin{equation} \label{S1goal} S_1  = N (\ln{X})^{2k} (\mathfrak{S}(\mathcal{G},\mathcal{H};N) I + o(1)). \end{equation}
If $g = g_1 \in \mathcal{G}, h = h_1 \in \mathcal{H}$ and
$$ \left( \sum_{i=2}^{k} (\alpha_i+\alpha'_i), \sum_{i=2}^{k} (\beta_i+\beta'_i) \right) \in (1-\varepsilon) \cdot \mathcal{D}, $$
we have
\begin{equation} \label{S2goal1} S_2  = N (\ln{X})^{2k+2} (\mathfrak{S}(\mathcal{G},\mathcal{H};N) J + o(1)). \end{equation}
If $g \notin \mathcal{G}, h \notin \mathcal{H}$ and
$$ \left( \sum_{i=1}^{k} (\alpha_i+\alpha'_i), \sum_{i=1}^{k} (\beta_i+\beta'_i) \right) \in (1-\varepsilon) \cdot \mathcal{D}, $$
we have
\begin{equation} \label{S2goal2} S_2  = N (\ln{X})^{2k} \left(\mathfrak{S}(\mathcal{G} \cup \{g\},\mathcal{H} \cup \{h\};N) I + o(1) \right). \end{equation}
Here, we write
$$ I = \prod_{i=1}^{k} \alpha_i \beta_i, \quad J = \alpha'_1 \beta'_1 \prod_{i=1}^{k} \alpha_i \beta_i, \quad \mathfrak{S}(\mathcal{G},\mathcal{H};N) = \mathfrak{S}(\mathcal{G} \cup (-\mathcal{H}-N))). $$
\end{lem}

\begin{proof}
We first prove \eqref{S1goal}. The proof of \eqref{S0goal} is similar. 

Let $\lambda_R(d)=\1_{d\leq R} \mu(d) \ln{R/d}$. Expanding $\Lambda_R$, we see that
$$ \prod_{i=1}^{k} \Lambda_{R_i}(n+h_i) \Lambda_{R'_i}(n+h_i) = \sum_{d_i,d'_i|n+h_i} \prod_{i=1}^{k} \lambda_{R_i}(d_i) \lambda_{R'_i}(d'_i). $$
We rewrite the sum on the right-hand side by setting $d=[d_1,d'_1,d_2,d'_2,\ldots,d_k,d'_k]$. This leads to
\begin{eqnarray}
\sum_{d_i,d'_i|n+h_i} \prod_{i=1}^{k} \lambda_{R_i}(d_i) \lambda_{R'_i}(d'_i) &=& \sum_{d \leq \prod_{i} R_i R'_i} \sum_{\substack{d_1,d'_1,\ldots,d_k,d'_k \\ [d_1,d'_1,\ldots,d_k,d'_k]=d \\ d_i,d'_i|n+h_i}} \prod_{i=1}^{k} \lambda_{R_i}(d_i) \lambda_{R'_i}(d'_i), \notag \\
&=& \sum_{d \leq \prod_{i} R_i R'_i} \sum_{\substack{l \in \Omega(d,\mathcal{H}) \\ n \equiv l \bmod d}} f(d,l;\mathcal{H}), \label{re type I}
\end{eqnarray}
where $\Omega(d;\mathcal{H}) \subseteq \mathbb{Z}/d\mathbb{Z}$ denotes
$$ \Omega(d;\mathcal{H}) = \{ l \in \mathbb{Z}/d\mathbb{Z} : \exists (d_1,d'_1,\ldots,d_k,d'_k) \in \mathbb{N}^{2k}, d=[d_1,d'_1,\ldots,d_k,d'_k], 1 \leq \forall i \leq k, [d_i,d'_i]|l+h_i\}, $$
and $f_R(d,l;\mathcal{H})$ denotes
$$ f_R(d,l;\mathcal{H}) = \sum_{\substack{d_1,d'_1,\ldots,d_k,d'_k \\ [d_1,d'_1,\ldots,d_k,d'_k]=d \\ d_i,d'_i|l+h_i}} \prod_{i=1}^{k} \lambda_{R_i}(d_i) \lambda_{R'_i}(d'_i). $$
We define $f_Q(d,l;\mathcal{G})$ similarly. Thus we have
$$ S_1 = \sum_{\substack{d_1 \leq \prod_{i} Q_i Q'_i \\ d_2 \leq \prod_{i} R_i R'_i}} \sum_{\substack{l_1 \in \Omega(d_1,\mathcal{G}) \\ l_2 \in \Omega(d_2,\mathcal{H})}} f_Q(l_1.d_1;\mathcal{G}) f_R(l_2,d_2;\mathcal{H}) \sum_{\substack{m+n=N \\ m \equiv l_1 \bmod d_1 \\ n \equiv l_2 \bmod d_2}} 1. $$

It is easy to see that $|\Omega(d,\mathcal{H})| \leq \tau(d)^{2k}$ and $|f_Q(d,l;\mathcal{H})|,|f_R(d,l;\mathcal{H})| \leq \tau(d)^{2k} (\ln{X})^{2k}$. 
We also note that 
$$ \sum_{\prod_{i} R_i R'_i<n \leq X} \left( \sum_{\substack{d \leq \prod_{i} R_i R'_i \\ n \pmod d \in \Omega(d;\mathcal{H})}} 1 \right)^2 \leq \sum_{X^{\varepsilon}<n \leq X} \left( \sum_{\substack{d_1,d'_1,\ldots,d_k,d'_k \\ [d_i,d'_i]|n+h_i \forall i}} 1 \right)^2 \leq \sum_{X^{\varepsilon}<n \leq X} \prod_{i} \tau(n+h_i)^4 \ll X^{1+o(1)}, $$
and if $l \in \Omega(d;\mathcal{H})$ one has 
\begin{equation} \label{(d,l)<} (d,l) = ([d_1,d'_1,\ldots,d_k,d'_k],l) \leq \prod_{i=1}^{k} ([d_i,d'_i],l) = \prod_{i=1}^{k} ([d_i,d'_i],-h_i) \leq \prod_{i=1}^{k} |h_i|, \end{equation}
which will be used later.

Thus, by Corollary \ref{11 level ext} with $Y=X^2$, we have
\begin{eqnarray*}
S_1 &=& \frac{N}{X^2} \sum_{\substack{d_1 \leq \prod_{i} Q_i Q'_i \\ d_2 \leq \prod_{i} R_i R'_i}} \sum_{\substack{l_1 \in \Omega(d_1;\mathcal{G}) \\ l_2 \in \Omega(d_2;\mathcal{H})}} f(l_1,d_1;\mathcal{G}) f(l_2,d_2;\mathcal{H}) \sum_{\substack{m \leq X^2 \\ m \equiv l_1 \bmod d_1 \\ N-m \equiv l_2 \bmod d_2}} 1 + O(X) \\
&=& \frac{N}{X^2} \sum_{m \leq X^2} \prod_{i=1}^{k} \Lambda_{Q_i}(m+g_i) \Lambda_{Q'_i}(m+g_i) \Lambda_{R_i}(N-m+h_i) \Lambda_{R'_i}(N-m+h_i) + O(X) \\
\end{eqnarray*}
for almost all $N \in (X/2,X]$ if $\sum_{i=1}^{k} (\alpha_i+\alpha'_i), \sum_{i=1}^{k} (\beta_i+\beta'_i) \leq 1- \varepsilon$.

Noting that $\Lambda_R(n)=\Lambda_R(-n)$, we have
$$ \sum_{m \leq X^2} \prod_{i=1}^{k} \Lambda_{Q_i}(m+g_i) \Lambda_{Q'_i}(m+g_i) \Lambda_{R_i}(N-m+h_i) \Lambda_{R'_i}(N-m+h_i) = X^2 (\ln{X})^{2k} (\mathfrak{S}(\mathcal{G},\mathcal{H};N) I + o(1)) $$
by Lemma \ref{GYlem}. This implies \eqref{S1goal}.

Our next goal is to show that
\begin{eqnarray*}
S_2 &=& \sum_{m+n=N} \Lambda(m+g_1) \Lambda(n+h_1) \prod_{i=1}^{k} \Lambda_{Q_i}(m+g_i) \Lambda_{Q'_i}(m+g_i) \Lambda_{R_i}(n+h_i) \Lambda_{R'_i}(n+h_i), \\
&=& X (\ln{X})^k (\mathfrak{S}(\mathcal{H}) J +o(1)) = N (\ln{X})^k (\mathfrak{S}(\mathcal{H}) \alpha_1 \beta_1 \prod_{i=1}^{k} \alpha_i \beta_i +o(1)).
\end{eqnarray*}
If $n+h_1 > X^{\alpha'} = R'_1 \geq R_1$ is prime, we see that 
$$ \Lambda_{R'_1}(n+h_1) = \sum_{\substack{d|n+h_1 \\ d \leq R'_1}} \mu(d) \ln{R'_1/d} = \ln{R'_1}, \quad \Lambda_{R_1}(n+h_1) = \ln{R_1}. $$
The same holds for the case $m+h_1 > Q'_1$ is prime. Thus, by discarding the terms for which at least one of the following holds: $n+h_1 \leq X^{\alpha'},m-h_1 \leq X^{\beta'}$, $n+h_1=p^a$ with $a>1$, $m-h_1=p^a$ with $a>1$, using $\Lambda_R(n) \ll X^{o(1)}$, we see that
\begin{multline*}
S_2 = (\ln{X})^2 \alpha_1 \alpha'_1 \beta_1 \beta'_1 \\
\times \sum_{m+n=N} \Lambda(m+g_1) \Lambda(n+h_1) \prod_{i=2}^{k} \Lambda_{Q_i}(m+g_i) \Lambda_{Q'_i}(m+g_i) \Lambda_{R_i}(n+h_i) \Lambda_{R'_i}(n+h_i) + O(X). 
\end{multline*}
Now the problem is reduced to showing that
\begin{multline}
\label{LL LQ}
\sum_{m+n=N} \Lambda(m+g_1) \Lambda(n+h_1) \prod_{i=2}^{k} \Lambda_{Q_i}(m+g_i) \Lambda_{Q'_i}(m+g_i) \Lambda_{R_i}(n+h_i) \Lambda_{R'_i}(n+h_i) \\
= N (\ln{X})^{2k-2} (\mathfrak{S}(\mathcal{H}) \prod_{i=2}^{k} \alpha_i \beta_i +o(1)). 
\end{multline}
Rewriting $m'=m+g_1, n'=n+h_1, g'_i = g_i - g_1, h'_i = h_i - h_1, N'=N+g_1+h_1$, the right-hand side becomes
$$ \sum_{m'+n'=N'} \Lambda(m') \Lambda(n') \prod_{i=2}^{k} \Lambda_{Q_i}(m'+g'_i) \Lambda_{Q'_i}(m'+g'_i) \Lambda_{R_i}(n'+h'_i) \Lambda_{R'_i}(n'+h'_i) + O(X). $$
Here, we used \eqref{LR bound} for the error term. 
Recalling \eqref{re type I}, this admits the expression
$$ \sum_{\substack{d_1 \leq \prod_{i=2}^{k} Q_i Q'_i \\ d_2 \leq \prod_{i=2}^{k} R_i R'_i}} \sum_{\substack{l_1 \in \Omega(d_1;\mathcal{G}') \\ l_2 \in \Omega(d_2;\mathcal{H}')}} f(l_1.d_1;\mathcal{G}') f(l_2,d_2;\mathcal{H}') \sum_{\substack{m'+n'=N' \\ m' \equiv l_1 \bmod d_1 \\ n' \equiv l_2 \bmod d_2}} \Lambda(m') \Lambda(n'), $$
where $\mathcal{G}',\mathcal{H}'$ denote $\{g'_2,\ldots,g'_k\},\{h'_2,\ldots,h'_k\}$ respectively. By \eqref{(d,l)<}, we have
$$ \sum_{d \leq X} \frac{1}{d} \sum_{l \in \Omega(d;\mathcal{H}')} (d,l) \leq \sum_{d \leq X} \frac{1}{d} \sum_{l \in \Omega(d;\mathcal{H}')} \prod_{i=2}^{k} |h_i-h_1| \leq (2 (\ln{X})^3)^{k-1} \sum_{d \leq X} \frac{\tau(d)^{2k}}{d} = X^{o(1)}. $$
The same holds for $\mathcal{G}'$. Therefore, by Corollary \ref{LL level ext}, the right-hand side of \eqref{LL LQ} becomes
\begin{eqnarray*}
&=& \sum_{\substack{d_1 \leq \prod_{i=2}^{k} Q_i Q'_i \\ d_2 \leq \prod_{i=2}^{k} R_i R'_i}} \sum_{\substack{l_1 \in \Omega(d_1;\mathcal{G}') \\ l_2 \in \Omega(d_2;\mathcal{H}')}} f(l_1.d_1;\mathcal{G}') f(l_2,d_2;\mathcal{H}') \sum_{\substack{m'+n'=N' \\ m' \equiv l_1 \bmod d_1 \\ n' \equiv l_2 \bmod d_2}} \Lambda_{X^{\varepsilon}}(m') \Lambda_{X^{\varepsilon}}(n') + O\left( \frac{X}{(\ln{X})^A} \right) \\
&=& \sum_{m'+n'=N'} \Lambda_{X^{\varepsilon}}(m') \Lambda_{X^{\varepsilon}}(n') \prod_{i=2}^{k} \Lambda_{Q_i}(m'+g'_i) \Lambda_{Q'_i}(m'+g'_i) \Lambda_{R_i}(n'+h'_i) \Lambda_{R'_i}(n'+h'_i) + O\left( \frac{X}{(\ln{X})^A} \right) \\
&=& \sum_{m+n=N} \Lambda_{X^{\varepsilon}}(m+g_1) \Lambda_{X^{\varepsilon}}(n+h_1) \prod_{i=2}^{k} \Lambda_{Q_i}(m+g_i) \Lambda_{Q'_i}(m+g_i) \Lambda_{R_i}(n+h_i) \Lambda_{R'_i}(n+h_i) + O\left( \frac{X}{(\ln{X})^A} \right) \\
&=& \sum_{m \leq N} \Lambda_{X^{\varepsilon}}(m+g_1) \Lambda_{X^{\varepsilon}}(N-m+h_1) \prod_{i=2}^{k} \Lambda_{Q_i}(m+g_i) \Lambda_{Q'_i}(m+g_i) \Lambda_{R_i}(N-m+h_i) \Lambda_{R'_i}(N-m+h_i) \\
&& + O\left( \frac{X}{(\ln{X})^A} \right)
\end{eqnarray*}
for almost all integers $N \in (X/2,X]$ if 
\begin{equation} \label{a+b<D 2} \left( \sum_{i=2}^{k} (\alpha_i+\alpha'_i), \sum_{i=2}^{k} (\beta_i+\beta'_i) \right) \in (1-\varepsilon) \cdot \mathcal{D}. \end{equation}
The last expression becomes
$$ N (\mathfrak{S}(\mathcal{H},\mathcal{G};N) \prod_{i=2}^{k} \alpha_i \beta_i + o(1)), $$
by Lemma \ref{GYlem}, since \eqref{a+b<D 2} implies that
$$ \varepsilon + \varepsilon + \sum_{i=2}^{k} (\alpha_i+\alpha'_i + \beta_i+\beta'_i) \leq 2\varepsilon + \frac{1}{2} + \frac{1}{3} < 1. $$
This gives \eqref{S2goal1}.

It remains to show that
\begin{eqnarray*}
S_2 &=& \sum_{m+n=N} \Lambda(n+h) \Lambda(m+g) \prod_{i=1}^{k} \Lambda_{Q_i}(m+g_i) \Lambda_{Q'_i}(m+g_i) \Lambda_{R_i}(n+h_i) \Lambda_{R'_i}(n+h_i) \\
&=& X (\ln{X})^{2k} (\mathfrak{S}(\mathcal{H},\mathcal{G};N) I +o(1)) = N (\ln{X})^{2k} (\mathfrak{S}(\mathcal{H},\mathcal{G};N) \prod_{i=1}^{k} \alpha_i \beta_i +o(1))
\end{eqnarray*}
for integers $g,h$ with $g \notin \mathcal{G}, h \notin \mathcal{H}, |g|,|h| \leq (\ln{X})^3$. This has already been established in \eqref{LL LQ} by replacing $k$ with $k+1$. 
\end{proof}

If each of $(\ln{X})^{O(1)}$ number of equations holds for all $N \in (X/2,X]$ with $O(X(\ln{X})^{-A})$ exceptions, then they all hold simultaneously with $O(X(\ln{X})^{-A+O(1)})$ exceptions. By adjusting $A$, the size of exceptional set is again bounded by $O(X(\ln{X})^{-A})$. We will use this simple fact several times without further mention.

\subsection{Average of singular series}

This section derives an asymptotic formula for the average of a singular series, generalizing Proposition 5.6 in \cite{tsuda2024small} and the result from Gallagher \cite{gallagher1976distribution}. We basically follow Gallagher's proof. 

\begin{lem}\label{gal local}
Let $p$ be a prime, and let $\mathcal{G}$ be a finite set of integers. Then, we have
\begin{equation} \label{sum H+G/p} \frac{1}{p^k} \sum_{h_1,\ldots,h_k \leq p} \left( 1 - \frac{1}{p}|(\{h_1,\ldots,h_k\} + \mathcal{G})/p\mathbb{Z}| \right) = \left( 1 - \frac{|\mathcal{G}/p\mathbb{Z}|}{p} \right)^k. \end{equation}
Here, $\mathcal{H} + \mathcal{G}$ denotes the Minkowski sum $\{h+g : h \in \mathcal{H}, g \in \mathcal{G} \}$. 
\end{lem}

\begin{proof}
Consider the quantity
$$ \sum_{\substack{h,h_1,\ldots,h_k \leq p \\ h \bmod p \notin (\{h_1,\ldots,h_k\}+{\mathcal{G})/p\mathbb{Z}}}} 1. $$
Summing over $h$, this becomes
$$ \sum_{h_1,\ldots,h_k \leq p} \sum_{\substack{h \leq p \\ h \bmod p \notin (\{h_1,\ldots,h_k\}+{\mathcal{G}})/p\mathbb{Z} }} 1 = \sum_{h_1,\ldots,h_k \leq p} \left( p - |(\{h_1,\ldots,h_k\}+\mathcal{G})/p\mathbb{Z}| \right). $$
On the other hand, summing over $h_1,\ldots,h_k$ first, this becomes
\begin{align*}
\sum_{h\leq p} \sum_{\substack{h_1,\ldots,h_k \leq p \\ h \bmod p \notin (\{h_i\} + \mathcal{G})/p\mathbb{Z} }} 1 &= \sum_{h\leq p} \sum_{\substack{h_1,\ldots,h_k \leq p \\ h_i \bmod p \notin \{h-g_1,\ldots,h-g_l\}/p\mathbb{Z} }} 1 \\
&= \sum_{h\leq p} \left( \sum_{\substack{h' \leq p \\ h' \bmod p \notin \{g_1,\ldots,g_l\}/p\mathbb{Z} }} 1 \right)^k = p \left( p - |\mathcal{G}/p\mathbb{Z}| \right)^k. 
\end{align*}
Comparing both expressions yields \eqref{sum H+G/p}. 
\end{proof}

From this, we expect that
\begin{align*}
& \mathbb{E}_{h_1,\ldots,h_k} \left( \mathfrak{S}(\{h_1,\ldots,h_k\}+\{0,N\}) \right) \\
=& \mathbb{E}_{h_1,\ldots,h_k} \left( \prod_{p} \left(1 - \frac{|(\{h_1,\ldots,h_k\}+\{0,N\})/p\mathbb{Z}|}{p}\right) \left(1 - \frac{1}{p}\right)^{-2k} \right) \\
\approx& \prod_{p} \mathbb{E}_{h_1,\ldots,h_k} \left(1 - \frac{|(\{h_1,\ldots,h_k\}+\{0,N\})/p\mathbb{Z}|}{p}\right) \left(1 - \frac{1}{p}\right)^{-2k} \\
=& \prod_{p} \left( 1 - \frac{|\{0,N\}/p\mathbb{Z}|}{p} \right)^k \left(1 - \frac{1}{p}\right)^{-2k} = \mathfrak{S}(N)^k. \\
\end{align*}
The remainder of this section is devoted to a justification of such a formula. \footnote{If $h_i$ range over a large interval such as $[0,N]$, this can also be obtained via Lemma \ref{GYlem}. }

\begin{lem}\label{S(H) av}
Let $X\geq 2, A, \varepsilon>0$, and let $k$ be a positive integer. Let $0<C_{-}<C_{+}$, and $H:\mathbb{N} \to \mathbb{R}$ be a function satisfying $(\ln{X})^{C_{-}} \ll H(N) \ll (\ln{X})^{C_{+}}$ for all $N \in (X/2,X]$. We have
\begin{equation} \label{gal eq goal} \sum_{\substack{1 \leq h_1,\ldots,h_k \leq H(N) \\ h_i \text{ distinct}}} \mathfrak{S}(\{h_1,\ldots,h_k\}+\{0,N\}) = (H(N) \mathfrak{S}(\{0,N\}))^k (1+ O(H(N)^{\varepsilon-1})). \end{equation}
for all even integers $N \in (X/2,X]$ with $O(X (\ln{X})^{-A})$ exceptions. 
\end{lem}

For a set of $k$ integers $\{h_1,\ldots,h_k\}$ with $h_1<\ldots<h_k$, the discriminant $\Delta(\{h_1,\ldots,h_k\})$ is defined by
$$ \Delta(\{h_1,\ldots,h_k\}) = \prod_{1 \leq i<j\leq k} (h_j - h_i). $$

\begin{proof}

We temporarily fix even integer $N \in (X/2,X]$ and set $H=H(N), \mathcal{G}=\{0,N\}$.

For a set of distinct integers $\mathcal{H}$ and a positive integer $k$, we define $S_{k}(s;\mathcal{H},\mathcal{G})$ by
\begin{equation} \label{def Sk} S_{k}(s;\mathcal{H},\mathcal{G}) = \prod_{p} \left( 1 - \frac{1}{p^s} \left( 1 - \frac{\left(1 - \frac{|(\mathcal{H}+\mathcal{G})/p\mathbb{Z}|}{p}\right)}{\left(1 - \frac{|\mathcal{G}/p\mathbb{Z}|}{p}\right)^{k}} \right) \right) = \sum_{q} \frac{a^{(k)}_q(\mathcal{H},\mathcal{G})}{q^{s}}. \end{equation}
If $p \nmid \Delta(\mathcal{H}+\mathcal{G})$, we see that $p \nmid f-f'$ for any distinct integers $f,f' \in \mathcal{H}+\mathcal{G}$. This implies $|(\mathcal{H}+\mathcal{G})/p\mathbb{Z}|=|\mathcal{H}+\mathcal{G}|=2|\mathcal{H}|$. We also have $|\mathcal{G}/p\mathbb{Z}|=|\mathcal{G}|=2$, since $\Delta(\mathcal{G})|\Delta(\mathcal{H}+\mathcal{G})$. Thus we have
\begin{eqnarray}
\label{akq b1}
a^{(k)}_p(\mathcal{H},\mathcal{G})
= \frac{\left(1 - \frac{|(\mathcal{H}+\mathcal{G})/p\mathbb{Z}|}{p}\right)}{\left(1 - \frac{|\mathcal{G}/p\mathbb{Z}|}{p}\right)^{k}} - 1 \ll_k
\begin{cases}
\frac{1}{p} & p|\Delta(\mathcal{H}+\mathcal{G}), \\
\frac{1}{p^2} & p\nmid \Delta(\mathcal{H}+\mathcal{G}). \\
\end{cases}
\end{eqnarray}
if $|\mathcal{H}|=k$. We also note that 
\begin{equation} \label{akq b2} a^{(k)}_q(\mathcal{H},\mathcal{G}) \ll \frac{\tau(q)^{O_k(1)}}{q}, \end{equation}
as long as $|\mathcal{H}|\leq k$.

The equation \eqref{gal eq goal} is equivalent to 
$$ \sum_{\substack{1 \leq h_1,\ldots,h_k \leq H \\ h_i \text{ distinct}}} S_{k}(0;\{h_1,\ldots,h_k\},\mathcal{G}) =H^k + O(H^{k-1+\varepsilon}). $$
We evaluate the left-hand side by truncating the series at $q = H$. 

We show that
\begin{equation} \label{Sk<H} \sum_{q \leq H} \sum_{\substack{1 \leq h_1,\ldots,h_k \leq H \\ h_i \text{ distinct}}} a^{(k)}_q(\{h_1,\ldots,h_k\},\mathcal{G}) = H^k + O\left(H^{k-1+o(1)} \sum_{\substack{d|N \\ d \leq H}} 1\right), \end{equation}
\begin{equation} \label{Sk>H} \sum_{q > H} a^{(k)}_q(\{h_1,\ldots,h_k\},\mathcal{G}) \ll H^{\varepsilon-1} \sigma_{-\varepsilon}(\Delta(\mathcal{H}+\mathcal{G}))^{O_k(1)}. \end{equation}

We show \eqref{Sk>H} first. Multiplying by $\left(\frac{q}{H}\right)^{1-\varepsilon} \geq 1$, we have
$$  \sum_{q > H} a^{(k)}_q(\{h_1,\ldots,h_k\},\mathcal{G}) \ll H^{\varepsilon-1} \sum_{q} q^{1-\varepsilon} |a^{(k)}_q(\{h_1,\ldots,h_k\},\mathcal{G})|. $$
By definition of $a^{(k)}_{q}$ and \eqref{akq b1}, we see that
\begin{eqnarray*}
\sum_{q} q^{1-\varepsilon} |a^{(k)}_q(\{h_1,\ldots,h_k\},\mathcal{G})| &=& \prod_{p} \left( 1 + p^{1-\varepsilon} |a^{(k)}_p(\{h_1,\ldots,h_k\},\mathcal{G})| \right) \\
&=& \prod_{p|\Delta(\mathcal{H}+\mathcal{G})} \left( 1 + O_k\left( \frac{1}{p^{\varepsilon}} \right) \right) \prod_{p\nmid \Delta(\mathcal{H}+\mathcal{G})} \left( 1 + O_k\left( \frac{1}{p^{1+\varepsilon}} \right) \right) \\
&\ll& \sigma_{-\varepsilon}(\Delta(\mathcal{H}+\mathcal{G}))^{O_k(1)}. \\
\end{eqnarray*}
This gives \eqref{Sk>H}.

By \eqref{akq b2}, we have
$$ \sum_{\substack{1 \leq h_1,\ldots,h_k \leq H \\ h_i \text{ distinct}}} a^{(k)}_q(\{h_1,\ldots,h_k\},\mathcal{G}) = \sum_{1 \leq h_1,\ldots,h_k \leq H} a^{(k)}_q(\{h_1,\ldots,h_k\},\mathcal{G}) + O\left( \frac{\tau(q)^{O_k(1)}}{q} H^{k-1} \right). $$
The main term of the right-hand side is handled by 
\begin{equation} \label{a+q=a} a^{(k)}_q(\{h_1+q,\ldots,h_k\},\mathcal{G})=a^{(k)}_q(\{h_1,\ldots,h_k\},\mathcal{G}), \end{equation}
which follows from definition of $a^{(k)}_p$ and multiplicativity of $a^{(k)}_q$. This gives
\begin{eqnarray*}
&& \sum_{1 \leq h_1,\ldots,h_k \leq H} a^{(k)}_q(\{h_1,\ldots,h_k\},\mathcal{G}) \\
&=& \frac{H}{q} \sum_{\substack{1 \leq h_2,\ldots,h_k \leq H \\ 1 \leq h_1 \leq q}} a^{(k)}_q(\{h_1,\ldots,h_k\},\mathcal{G}) + O\left( \sum_{\substack{1 \leq h_2,\ldots,h_k \leq H \\ 1 \leq h_1 \leq q}} |a^{(k)}_q(\{h_1,\ldots,h_k\},\mathcal{G})| \right) \\
&=& \left( \frac{H}{q} \right)^2 \sum_{\substack{1 \leq h_3,\ldots,h_k \leq H \\ 1 \leq h_1,h_2 \leq q}} a^{(k)}_q(\{h_1,\ldots,h_k\},\mathcal{G}) + O\left( \frac{H}{q} \sum_{\substack{1 \leq h_3,\ldots,h_k \leq H \\ 1 \leq h_1,h_2 \leq q}} |a^{(k)}_q(\{h_1,\ldots,h_k\},\mathcal{G})| \right) \\
&\vdots& \\
&=& \left( \frac{H}{q} \right)^k \sum_{1 \leq h_1,\ldots,h_k \leq q} a^{(k)}_q(\{h_1,\ldots,h_k\},\mathcal{G}) + O\left( \left( \frac{H}{q} \right)^{k-1} \sum_{1 \leq h_1,\ldots,h_k \leq q} |a^{(k)}_q(\{h_1,\ldots,h_k\},\mathcal{G})| \right) \\
&=& \left( \frac{H}{q} \right)^k A^{(k)}_q(\mathcal{G}) + O\left( \left( \frac{H}{q} \right)^{k-1} B^{(k)}_q(\mathcal{G}) \right). 
\end{eqnarray*}
where 
$$ A^{(k)}_q(\mathcal{G}) = \sum_{1 \leq h_1,\ldots,h_k \leq q} a^{(k)}_q(\{h_1,\ldots,h_k\},\mathcal{G}), \quad B^{(k)}_q(\mathcal{G}) = \sum_{1 \leq h_1,\ldots,h_k \leq q} |a^{(k)}_q(\{h_1,\ldots,h_k\},\mathcal{G})|. $$
Suppose $q=rs$ with $(r,s)=1$, then rewriting $h_i=rh'_i+sh''_i$, we see $A^{(k)}_q(\mathcal{G})=A^{(k)}_r(\mathcal{G})A^{(k)}_s(\mathcal{G})$, by \eqref{a+q=a}. The same holds for $B^{(k)}_q(\mathcal{G})$. Thus, $A^{(k)}_q(\mathcal{G})$ and $B^{(k)}_q(\mathcal{G})$ are multiplicative with respect to $q$. 
By Lemma \eqref{gal local}, we have
$$ A^{(k)}_p(\mathcal{G}) = \sum_{1 \leq h_1,\ldots,h_k \leq p} - \left( 1 - \frac{\left(1 - \frac{|(\mathcal{H}+\mathcal{G})/p\mathbb{Z}|}{p}\right)}{\left(1 - \frac{|\mathcal{G}/p\mathbb{Z}|}{p}\right)^{k}} \right) = 0. $$
By \eqref{akq b1} and \eqref{akq b2}, we have
$$ B^{(k)}_p(\mathcal{G}) \ll_k \sum_{\substack{1 \leq h_1,\ldots,h_k \leq p \\ |\{h_1,\ldots,h_k\}|=k}} \frac{(\Delta(\{h_1,\ldots,h_k\}+\mathcal{G}),p)}{p^2} + \sum_{\substack{1 \leq h_1,\ldots,h_k \leq p \\ |\{h_1,\ldots,h_k\}|<k}} \frac{1}{p} \ll_{k} p^{k-2} (p,N) $$
Thus, we see that $A^{(k)}_q(\mathcal{G}) = \1_{q=1}$ and $B^{(k)}_q(\mathcal{G}) \ll q^{k-2} \tau(q)^{O_k(1)} (q,N)$.

Putting everything together, we have
$$ \sum_{q \leq H} \sum_{\substack{1 \leq h_1,\ldots,h_k \leq H \\ h_i \text{ distinct}}} a^{(k)}_q(\{h_1,\ldots,h_k\},\mathcal{G}) = H^k + O\left(H^{k-1} \sum_{q \leq H} \frac{(q,N)}{q} \tau(q)^{O_k(1)} \right) $$
The error term is bounded by
$$ H^{k-1+o(1)} \sum_{q \leq H} \frac{(q,N)}{q} \leq H^{k-1+o(1)} \sum_{\substack{d|N \\ d \leq H}} d \sum_{\substack{q \leq H \\ d|q}} \frac{1}{q} \ll H^{k-1+o(1)} \sum_{\substack{d|N \\ d \leq H}} 1. $$
This gives \eqref{Sk<H}. 

Therefore, we have
\begin{multline*}
\sum_{\substack{1 \leq h_1,\ldots,h_k \leq H \\ h_i \text{ distinct}}} S_{k}(0;\{h_1,\ldots,h_k\},\{0,N\}) = H(N)^k \\
+ O\left( H(N)^{k-1+\varepsilon} \max_{\substack{\mathcal{H} \subseteq [0,(\ln{X})^{C_+}] \\ |\mathcal{H}| = k}} \sigma_{-\varepsilon}(\Delta(\mathcal{H}+\{0,N\}))^{O_k(1)} + H(N)^{k-1+o(1)} \sum_{\substack{d|N \\ d \leq H(N)}} 1 \right) 
\end{multline*}

It remains to show that the error term is bounded by $H(N)^{k-1+2\varepsilon}$ for almost all $N \in (X/2,X]$. 

We consider the second term first. Let $b$ be a positive integer, and consider the $b^2$-th moment. We have
$$ \sum_{X/2<N\leq X} \left(\sum_{\substack{d|N \\ d \leq (\ln{X})^{C_+}}} 1\right)^{b^2} \ll \sum_{d_1,\ldots,d_{b^2} \leq (\ln{X})^{C_+}} \frac{X}{[d_1,\ldots,d_{b^2}]} \leq X \sum_{d \leq (\ln{X})^{b^2 C_+}} \frac{\tau(d)^{b^2}}{d} \ll X \ln{X} $$
On the other hand, the left hand side is
$$ \geq \sum_{\substack{X/2<N\leq X \\ \sum_{\substack{d|N \\ d \leq H(N)}} 1>(\ln{X})^{1/b}}} \left(\sum_{\substack{d|N \\ d \leq H(N)}} 1\right)^{b^2} \geq (\ln{X})^b \sum_{\substack{X/2<N\leq X \\ \sum_{\substack{d|N \\ d \leq H(N)}} 1>(\ln{X})^{1/b}}} 1. $$
Thus, we see that for all integers $N \in (X/2,X]$ with $O(X(\ln{X})^{-b+1})$ exceptions, $\sum_{\substack{d|n \\ d \leq H(N)}} 1 \leq (\ln{X})^{1/b}$ holds. 

We consider the first term. Using simple inequality $\sigma_{-\varepsilon}(mn) \leq \sigma_{-\varepsilon}(m)\sigma_{-\varepsilon}(n)$ and $\sigma_{-\varepsilon}(n) \leq \tau(n) \ll n^{o(1)}$, we have
\begin{eqnarray*}
\sigma_{-\varepsilon}(\Delta(\{h_1,\ldots,h_k\}+\{0,N\})) &\ll& (\ln{X})^{o(1)} \prod_{1 \leq i \neq j \leq k} \sigma_{-\varepsilon}(N-h_i+h_j)\\
&\leq& (\ln{X})^{o(1)} \max_{|h| \leq (\ln{X})^{C_{+}}} \sigma_{-\varepsilon}(N+h)^{k^2} \\
\end{eqnarray*}
Thus the claim follows from next assertion: \par
For any given $B,C>0$, $\sigma_{-\varepsilon}(n)^{C} \ll (\ln{X})^{1/B}$ holds for all $n \in (X/2,X]$ but for $O(X (\ln{X})^{-B})$ exceptions. \par
This easily follows from
\begin{eqnarray*}
\sum_{\substack{X/2 < n \leq X \\ (\ln{X})^{1/B}<\sigma_{-\varepsilon}(n)^{C}}} (\ln{X})^B &\leq& \sum_{X/2 < n \leq X} \sigma_{-\varepsilon}(n)^{CB^2} \\
&\ll& \sum_{X/2 < n \leq X} \sigma_{-\varepsilon/2}(n)\ll \sum_{d \leq X} \frac{1}{d^{\varepsilon/2}} \left( \frac{X}{d} \right) \ll X.
\end{eqnarray*}

\end{proof}

Hereafter, we write 
$$ \mathfrak{S}(\mathcal{H};N) \coloneqq \mathfrak{S}(\mathcal{H}+\{0,N\})=\mathfrak{S}(-\mathcal{H},\mathcal{H};N)=\mathfrak{S}(\mathcal{H},-\mathcal{H};N)=\mathfrak{S}(\mathcal{H} \cup (\mathcal{H}-N)). $$

\section{Bombieri-Davenport method}\label{BD}

Following \cite{goldston2007higher3}, we now consider a lower bound for the second moment
$$ M_2 = \sum_{m+n=N} \left( \sum_{h \leq H} \Lambda(m+h) \Lambda(n-h) \right)^2, $$
where $H=c \mathfrak{S}(N)^{-1} (\ln{N})^2$ with $c>0$.

We use the trivial inequality:
\begin{equation} \label{gy ineq} 0 \leq \sum_{m+n=N} \left( \sum_{h \leq H} \Lambda(m+h) \Lambda(n-h)  - \sum_{h \leq H} (\Lambda(m+h) \Lambda(n-h))^{\sharp} \right)^2. \end{equation}
as an analogue of \eqref{BD start}. Here, $(\Lambda(m+h) \Lambda(n-h))^{\sharp}$ is a suitable approximation to $\Lambda(m+h) \Lambda(n-h)$, which we now take to be
\begin{equation} \label{LLsharp ch}(\Lambda(m) \Lambda(n))^{\sharp} = \Lambda_{X_2}(m) \Lambda_{X_3}(n) + \Lambda_{X_3}(m) \Lambda_{X_2}(n) - \Lambda_{X_3}(m) \Lambda_{X_3}(n), 
\end{equation}
where $X_r = X^{(1-\varepsilon)/r}$. 

Expanding the square in \eqref{gy ineq}, we see that
\begin{eqnarray*}
M_2 &\geq& 2 \sum_{h,h' \leq H} \sum_{m+n=N} \Lambda(m+h) \Lambda(n-h) (\Lambda(m+h') \Lambda(n-h'))^{\sharp} \\ 
&-& \sum_{h,h' \leq H} \sum_{m+n=N} (\Lambda(m+h) \Lambda(n-h))^{\sharp}  (\Lambda(m+h') \Lambda(n-h'))^{\sharp} \\
&=& 2 \sum_{\substack{h,h' \leq H \\ h \neq h'}} S^{\sharp}(h,h';N) + 2 \sum_{h \leq H} S^{\sharp}(h;N) - \sum_{\substack{h,h' \leq H \\ h \neq h'}} S^{\sharp \sharp}(h,h';N) - \sum_{h \leq H} S^{\sharp \sharp}(h;N)
\end{eqnarray*}
where we define
$$ S^{\sharp}(h,h';N) = \sum_{m+n=N} \Lambda(m+h) \Lambda(n-h) (\Lambda(m+h') \Lambda(n-h'))^{\sharp}, \quad S^{\sharp}(h;N) = S^{\sharp}(h,h;N) $$
$$ S^{\sharp \sharp}(h,h';N) = \sum_{m+n=N} (\Lambda(m+h) \Lambda(n-h))^{\sharp} (\Lambda(m+h') \Lambda(n-h'))^{\sharp}, \quad S^{\sharp \sharp}(h;N) = S^{\sharp \sharp}(h,h;N) $$
for integers $h,h'$. 

By Lemma \ref{mainlem} \eqref{S2goal2} $k=1$ and \eqref{S0goal} $k=2$, we have
$$ S^{\sharp}(h,h';N), S^{\sharp \sharp}(h,h';N) = N (\mathfrak{S}(\{h,h'\};N) + o(1)), $$
for distinct positive integers $h,h' \leq H$ and almost all $N$. Note that the little-$o$ term here is independent of $h,h'$. We also have
\begin{eqnarray*}
S^{\sharp}(h;N) &=& N (\mathfrak{S}(\{h\};N) \left( \frac{1}{2} \cdot \frac{1}{3} + \frac{1}{3} \cdot \frac{1}{2} - \frac{1}{3} \cdot \frac{1}{3} \right) (1-\varepsilon)^2 + o(1)) (\ln{X})^2, \\
&=& N (\mathfrak{S}(\{h\};N) \mathrm{meas}((1-\varepsilon)\mathcal{D}) + o(1)) (\ln{X})^2,
\end{eqnarray*}
\begin{eqnarray*}
S^{\sharp \sharp}(h;N) &=& N (\mathfrak{S}(\{h\};N) \\
&& \times \Bigg( \Bigg. \left( \min\left(\frac{1}{2},\frac{1}{2}\right) \min\left(\frac{1}{3},\frac{1}{3}\right) + \min\left(\frac{1}{2},\frac{1}{3}\right) \min\left(\frac{1}{3},\frac{1}{2}\right) - \min\left(\frac{1}{2},\frac{1}{3}\right) \min\left(\frac{1}{3},\frac{1}{3}\right) \right) \\
&& + \left( \min\left(\frac{1}{3},\frac{1}{2}\right) \min\left(\frac{1}{2},\frac{1}{3}\right) + \min\left(\frac{1}{3},\frac{1}{3}\right) \min\left(\frac{1}{2},\frac{1}{2}\right) - \min\left(\frac{1}{3},\frac{1}{3}\right) \min\left(\frac{1}{2},\frac{1}{3}\right) \right) \\
&& - \left( \min\left(\frac{1}{3},\frac{1}{2}\right) \min\left(\frac{1}{3},\frac{1}{3}\right) + \min\left(\frac{1}{3},\frac{1}{3}\right) \min\left(\frac{1}{3},\frac{1}{2}\right) - \min\left(\frac{1}{3},\frac{1}{3}\right) \min\left(\frac{1}{3},\frac{1}{3}\right) \right) \Bigg. \Bigg) \\
&& \times (1-\varepsilon)^2 + o(1)) (\ln{X})^2, \\
&=& N (\mathfrak{S}(\{h\};N) \left( \frac{1}{2} \cdot \frac{1}{3} + \frac{1}{3} \cdot \frac{1}{2} - \frac{1}{3} \cdot \frac{1}{3} \right) (1-\varepsilon)^2 + o(1)) (\ln{X})^2, \\
&=& N (\mathfrak{S}(\{h\};N) \mathrm{meas}((1-\varepsilon)\mathcal{D}) + o(1)) (\ln{X})^2,
\end{eqnarray*}
for almost all $N$, by Lemma \ref{mainlem} \eqref{S2goal1} $k=1$ and \eqref{S1goal} $k=1$. 

Putting everything together, we have
$$ M_2 \geq N\left( \sum_{\substack{h,h' \leq H \\ h \neq h'}} \mathfrak{S}(\{h,h'\};N) + \mathrm{meas}((1-\varepsilon)\mathcal{D}) (\ln{X})^2 \sum_{h \leq H} \mathfrak{S}(\{h\};N) + o(H^2+H(\ln{X})^2) \right) $$
for almost all even $N \in (X/2,X]$.

Noting that $H\mathfrak{S}(N)=c (\ln{N})^2$, by Lemma \ref{S(H) av}, we obtain the following proposition. 
\begin{prop}
Let $\varepsilon>0$ be sufficiently small and $A>0$. Let $c>0$ and $H=c \mathfrak{S}(N)^{-1} (\ln{N})^2$. We have
$$ M_2 \geq N (\ln{X})^4 \left( c^2 + c \cdot \mathrm{meas}(\mathcal{D}) + O(\varepsilon) \right) $$
for almost all even integers $N \in (X/2,X]$ with $O(X (\ln{X})^{-A})$ exceptions. 
\end{prop}

Since
\begin{eqnarray*}
M_2 &=& \sum_{\substack{h,h' \leq H \\ h \neq h'}} \sum_{m+n=N} \Lambda(m+h) \Lambda(n-h) \Lambda(m+h') \Lambda(n-h') + \sum_{h \leq H} \sum_{m+n=N} \Lambda(m+h)^2 \Lambda(n-h)^2 \\
&=& \sum_{\substack{h,h' \leq H \\ h \neq h'}} \sum_{\substack{m+n=N \\ N^{1-\epsilon} < m,n}} \Lambda(m+h) \Lambda(n-h) \Lambda(m+h') \Lambda(n-h') + H \sum_{m+n=N} \Lambda(m)^2 \Lambda(n)^2 + O(N) \\
&\leq& \sum_{\substack{h,h' \leq H \\ h \neq h'}} \sum_{\substack{m+n=N \\ N^{1-\epsilon} < m,n}} \Lambda(m+h) \Lambda(n-h) \Lambda(m+h') \Lambda(n-h') + N (\ln{X})^2 H \mathfrak{S}(N) + O(N),
\end{eqnarray*}
holds for almost all $N$, we now have
\begin{equation} \label{LLLL > NH} \sum_{\substack{h,h' \leq H \\ h \neq h'}} \sum_{\substack{m+n=N \\ N^{1-\epsilon} < m,n}} \Lambda(m+h) \Lambda(n-h) \Lambda(m+h') \Lambda(n-h') \geq N (\ln{X})^4 \left( c^2 - c (1- \mathrm{meas}(\mathcal{D})) + O(\varepsilon) \right) . \end{equation}
This leads to
\begin{equation} \label{7/9} \Xi \leq 1 - \mathrm{meas}(\mathcal{D}) = 1 - \left( \frac{1}{2} \cdot \frac{1}{3} + \frac{1}{3} \cdot \frac{1}{2} - \frac{1}{3} \cdot \frac{1}{3} \right) = \frac{7}{9} = 0.777... \end{equation}

\begin{rem}
If we take 
$$ (\Lambda(m) \Lambda(n))^{\sharp} = \sum_{d_1,d_2} \mu(d_1) \mu(d_2) F\left( \frac{\ln{d_1}}{\ln{X}}, \frac{\ln{d_2}}{\ln{X}} \right) $$
with a function $F$ that admits the representation
$$ F(u,v) = \iint_{\substack{u \leq s \\ v \leq t \\ (u,v) \in (1-\varepsilon) \cdot \mathcal{D}}} f(s,t) ds dt $$
for some piecewise smooth function $f$, the resulting bound on $M_2$ becomes
\begin{eqnarray*}
M_2 &\geq& N (\ln{X})^4 \left( c^2 (2f(0,0) - f(0,0)^2) + c \iint_{\substack{u \leq s \\ v \leq t \\ (u,v) \in (1-\varepsilon) \cdot \mathcal{D}}} (2 f(s,t)- f(s,t)^2) ds dt + O(\varepsilon) \right) \\
&=& N (\ln{X})^4 \left( c^2 + c \iint_{\substack{u \leq s \\ v \leq t \\ (u,v) \in (1-\varepsilon) \cdot \mathcal{D}}} ds dt - c^2 (f(0,0)-1)^2 - \iint_{\substack{u \leq s \\ v \leq t \\ (u,v) \in (1-\varepsilon) \cdot \mathcal{D}}} (f(s,t)-1)^2 ds dt + O(\varepsilon) \right).
\end{eqnarray*}
A similar argument will appear in the next section; the present bound for $M_2$ can be justified by adapting that method. 
Thus, the bound becomes optimal when $f(s,t)=1$, leading to the choice \eqref{LLsharp ch}. 
Alternatively, one may use
$$ (\Lambda(m) \Lambda(n))^{\sharp} = \sum_{(q_1,q_2) \in X^{(1-\varepsilon) \cdot \mathcal{D}}} \frac{\mu(q_1)}{\varphi(q_1)} \frac{\mu(q_2)}{\varphi(q_2)} c_{q_1}(m) c_{q_2}(n). \\
$$
 
\end{rem}

\subsection{Proof of Theorem \ref{Xi<}}

Following Tsuda \cite{tsuda2024small}, one can obtain further improvement on \eqref{7/9} by means of the upper bound sieve. 

Let $S(h,h';N)$ denote
$$ S(h,h';N) = \sum_{m+n=N} \Lambda(m+h) \Lambda(n-h) \Lambda(m+h') \Lambda(n-h'). $$

The Rosser Iwaniec upper bound sieve with well-factorable weights gives the following bound. 
\begin{lem}
Let $|h|,|h'| \leq (\ln{X})^3,h \neq h'$ and $A>0$. We have
$$ \sum_{\substack{m+n=N \\ N^{1-\epsilon} < m,n}} \Lambda(m+h) \Lambda(n-h) \Lambda(m+h') \Lambda(n-h') \leq 16 N \mathfrak{S}(\{h,h'\};N) (1+\varepsilon) $$
for almost all even integers $N \in (X/2,X]$ with $O(X (\ln{X})^{-A})$ exceptions. 
\end{lem}

\begin{proof}
This is Theorem 2.2 of \cite{tsuda2024small}. 
\end{proof}

\begin{rem}
It may be observed that the constant $16$ arises as
$$ 16 = \frac{1}{\mathrm{meas}(\frac{1}{2} \cdot \mathcal{D}^*)}, $$
where $\mathcal{D}^*$ is defined in \eqref{D star}. We note that Theorem \ref{LL level} gives a bound in which $16$ is replaced by $144/7=20.57...$. 

Let $\mathcal{D}' \subset [0,1]^2$ be a closed set satisfying $\mathcal{D}'+\mathcal{D}' \subseteq \mathcal{D}=([0,1/2]\times[0,1/3]) \cup ([0,1/3]\times[0,1/2])$ and define
$$ (\Lambda(m) \Lambda(n))^{\sharp} = \sum_{d_1,d_2} \mu(d_1) \mu(d_2) \iint_{\substack{\ln{d_1}/\ln{X} \leq s \\ \ln{d_2}/\ln{X} \leq t \\ (u,v) \in (1-\varepsilon) \cdot \mathcal{D}'}} ds dt. $$
Then we have
$$ \Lambda(m) \Lambda(n) \leq \left(\frac{(\Lambda(m) \Lambda(n))^{\natural}}{\mathrm{meas}((1-\varepsilon) \cdot \mathcal{D}')}\right)^2 $$
unless either $m$ or $n$ is very small, or is a prime power. 

This, together with Lemma \ref{mainlem}, gives a bound in which $16$ is replaced by $1/\mathrm{meas}((1-\varepsilon) \cdot \mathcal{D}')$ under some natural conditions on $\mathcal{D}'$. 
If we take $\mathcal{D}'=\{(u,v) \in [0,1/4] \times [0,1/4] : u+v \leq 1/3 \}$, we then have $\mathrm{meas}(\mathcal{D}')=7/144$. 
\end{rem}

Let $\kappa>\frac{7}{9}>\lambda>0$ and $H_+= \kappa \mathfrak{S}(N)^{-1} (\ln{N})^2, H_- = \lambda \mathfrak{S}(N)^{-1} (\ln{N})^2$. By \eqref{LLLL > NH}, we have
$$ \sum_{\substack{h,h' \leq H_+ \\ h \neq h'}} \sum_{\substack{m+n=N \\ N^{1-\epsilon} < m,n}} \Lambda(m+h) \Lambda(n-h) \Lambda(m+h') \Lambda(n-h') \geq N (\ln{X})^4  (\kappa^2 - \frac{7}{9} \kappa + O(\varepsilon) ). $$
The left-hand side is bounded by
\begin{eqnarray*}
&& \left( \sum_{\substack{h,h' \leq H_+ \\ h \neq h' \\ |h-h'| \leq H_-}} + \sum_{\substack{h,h' \leq H_+ \\ h \neq h' \\ |h-h'| > H_-}} \right) \sum_{\substack{m+n=N \\ N^{1-\epsilon} < m,n}} \Lambda(m+h) \Lambda(n-h) \Lambda(m+h') \Lambda(n-h') \\
&\leq& \sum_{\substack{h \neq h' \\ |h-h'| \leq H_-}} \sum_{\substack{m+n=N \\ N^{1-\epsilon} < m,n}} \Lambda(m+h) \Lambda(n-h) \Lambda(m+h') \Lambda(n-h') + \sum_{\substack{h,h' \leq H_+ \\ h \neq h' \\ |h-h'| > H_-}} 16 N \mathfrak{S}(\{h,h'\};N) (1+O(\varepsilon)) \\
\end{eqnarray*}
We now prove that
\begin{equation} \sum_{\substack{h,h' \leq H_+ \\ h \neq h' \\ |h-h'| > H_-}} \mathfrak{S}(\{h,h'\};N) = ((\kappa - \lambda)^2 + o(1))(\ln{X})^4 \quad \left( \approx \sum_{\substack{h,h' \leq H_+ \\ h \neq h' \\ |h-h'| > H_-}} \mathfrak{S}(N)^2 \right)  \label{S(H)_cut} \end{equation}
for almost all $N\in (X/2,X]$. To see that this follows from Lemma \ref{S(H) av}, we first observe that $\mathfrak{S}(\{h,h'\};N)$ depends only on $h-h'$ by the definition. 
Thus, for any fixed $c$ and $H=c \mathfrak{S}(N)^{-1} (\ln{N})^2$, we have
\begin{equation} \label{S(H)_cc} \sum_{h} \mathfrak{S}(\{h,0\};N) \max(H-|h|,0) = \sum_{h,h' \leq H} \mathfrak{S}(\{h,h'\};N) = (c^2+O((\ln{X})^{-1})) (\ln{X})^4 \end{equation}
for almost all $N$. Similarly, we have
$$ \sum_{\substack{h,h' \leq H_+ \\ h \neq h' \\ |h-h'| > H_-}} \mathfrak{S}(\{h,h'\};N) = \sum_{h} \mathfrak{S}(\{h,0\};N) \max(H_{+} -|h|,0) \1_{|h|>H_-}. $$
Noting that 
$$ \max(H_{+} -|h|,0) \1_{|h|>H_-} = \max(H_{+} -|h|,0) - \max(H_{-} -|h|,0) - (H_+-H_-) \1_{|h| \leq H_-}, $$
we now find \eqref{S(H)_cut} follows from
\begin{equation} \label{S(H)_1} \sum_{h} \mathfrak{S}(\{h,0\};N) \1_{|h| \leq H_-} = 2 H_- \mathfrak{S}(N)^2 (1+o(1)). \end{equation}
Let $\delta=(\ln{X})^{-1/2}$. The inequality
$$ \max(H_{-} -|h|,0) - \max((1-\delta)H_{-} -|h|,0) \leq \delta H_- \1_{|h| \leq H_-} \leq \max((1+\delta)H_{-} -|h|,0) - \max(H_{-} -|h|,0) $$
and \eqref{S(H)_cc} yields \eqref{S(H)_1}.

Therefore, we obtain
$$ \sum_{\substack{h \neq h' \\ |h-h'| \leq H_-}} \sum_{\substack{m+n=N \\ N^{1-\epsilon} < m,n}} \Lambda(m+h) \Lambda(n-h) \Lambda(m+h') \Lambda(n-h') \geq N (\ln{X})^4 \left( \kappa^2 - \frac{7}{9} \kappa - 16 (\kappa - \lambda)^2 + O(\varepsilon) \right)$$
The right-hand side is positive if $\varepsilon$ is small and
$$ \kappa^2 - \frac{7}{9} \kappa > 16 (\kappa - \lambda)^2. $$
This holds when
$$ \lambda > \min_{7/9 \leq \kappa'} \left( \kappa' - \frac{\sqrt{\kappa'^2 - \frac{7}{9} \kappa'}}{4} \right) = \frac{7}{72} (4+\sqrt{15}) $$
and the minimum attained at $\kappa=\kappa'$. 

\par

Further improvement might be possible if one were able to show a Bombieri-Friedlander-Iwaniec-type estimate of the form
$$ \sum_{d \leq X^{4/7 - \varepsilon}} \lambda_d \left( \sum_{\substack{n \leq X \\ n \equiv h \bmod d \\ n \equiv b \bmod q}} (\Lambda(n) - \Lambda_R(n)) \right) \ll \frac{X}{(\ln{X})^A}. $$
for $|h|,|b| \ll (\ln{X})^A$ and a well-factorable sequence $\lambda_d$ of level $X^{4/7 - \varepsilon}$. Moreover, one might obtain further improvement by employing the results on the level of distribution for upper-bound well-factorable linear sieve weights \cite{maynard2025primes} \cite{pascadi2025exponents}. 
However, any such improvement would be very small, and lies beyond the scope of this paper.

\section{GPY Maynard-Tao method}\label{MT}

Let $\mathcal{H}=\{h_1,\ldots,h_k\}$ be a set of $k$ integers, and define
\begin{equation} \label{bounded S} S = S(\mathcal{H},N,w) = \sum_{\substack{m+n=N \\ N^{1-\epsilon} < m,n}} \left( \sum_{\ell=1}^{k} \Lambda'(m+h_{\ell}) \Lambda'(n-h_{\ell}) - (\ln{N})^2 \right) w(m,n;\mathcal{H}) \end{equation}
where $\Lambda'(n)$ denotes $(\ln{n}) \1_{\mathbb{P}}(n)$. If we can show that there exists a non-negative weight $w(m,n;\mathcal{H})$ and $\mathcal{H}$ such that $S(\mathcal{H},N,w)>0$ holds for almost all even integers $N$, this implies 
$$ \Xi^{*} \leq \max_{h,h' \in \mathcal{H}} (h-h'). $$
Next, we let $\gamma>0, H=\gamma \mathfrak{S}(N)^{-1} (\ln{N})^2$ and consider
$$ S = \sum_{\substack{m+n=N \\ N^{1-\epsilon} < m,n}} \sum_{\substack{(h_1,\ldots,h_k) \in [0,H]^k \\ 1 \leq h_i \leq H \\ h_i \text{distinct}}} \left( \sum_{h \leq H} \Lambda'(m+h) \Lambda'(n-h) - (\ln{N})^2 \right) w(m,n;\mathcal{H}). $$
Clearly, if we could show that $S>0$ for almost all even integers $N$ with a non-negative weight $w(m,n,\mathcal{H})$, this implies
$$ \Xi \leq \gamma. $$
A similar argument is used in the work of Goldston, Pintz, and Yildirim \cite{goldston2009primes}.

We investigate such expressions using the Maynard \cite{maynard2015small}, Tao (see \cite{polymath2014variants} or his blog) type choice of weight
$$ w(m,n;\mathcal{H}) = \left( \sum_{\substack{d_i|n+h_i, 1 \leq i \leq k \\ e_i|m-h_i, 1 \leq i \leq k}} \lambda_{d_1,\ldots,d_k,e_1,\ldots,e_k} \right)^2 $$
where $\lambda$ will be chosen later. It is sometimes more convenient to write this as
$$ \left( \sum_{\substack{d_i|n+h_i, 1 \leq i \leq 2k}} \lambda_{d_1,\ldots,d_{2k}} \right)^2 $$
where we define $h_i=N-h_{i-k}$ for $k<i\leq 2k$. We use Tao's choice of $\lambda$:
\begin{equation} \label{Tao lambda} \lambda_{d_1,\ldots,d_{2k}} = \mu{(d_1)} \cdots \mu{(d_{2k})} F\left(\frac{\ln{d_1}}{\ln{X}},\ldots,\frac{\ln{d_{2k}}}{\ln{X}}\right). \end{equation}
for a suitable function $F:[0,1]^{2k} \to \mathbb{R}$. However, both Maynard's and Tao's approaches to evaluating $S$ somewhat rely on the $W$-trick, which becomes an obstruction in our setting, since the discriminant
$$ \Delta(\mathcal{H} \cup (\mathcal{H}-N)) = \prod_{1 \leq i < j \leq 2k} (h_i-h_j) $$
may have large prime factors. Even if one were able to carry out the calculation, the resulting bound for the exceptional set become worse, since $W$-trick may destroy the singular series. 
For these reasons, we adopt a method distinct from those of both Maynard and Tao, although still close to \cite{polymath2014variants}.

We call a function $f:[0,1]^n \to \mathbb{R}$ a step function if $f$ can be written as a finite linear combination of indicator functions of intervals of the form $[a_1,b_1) \times \cdots \times [a_n,b_n)$. 

\begin{lem}\label{gpymt mainlem}
Let $X \geq 2, \varepsilon>0,A>0$ and $\mathcal{G}=\{g_1,\ldots,g_k\}, \mathcal{H}=\{h_1,\ldots,h_k\}$ be a set of integers with $|h_i|,|g_i| \leq (\ln{X})^3, 1 \leq i \leq k$. 

Let $\mathcal{R} \subseteq [0,1]^{2k}$ be a closed set such that
\begin{eqnarray}
\mathcal{R}+\mathcal{R} &=& \{ (s_1+t_1,\ldots,s_{2k}+t_{2k}) : (s_1,\ldots,s_{2k}),(t_1,\ldots,t_{2k}) \in \mathcal{R} \} \notag \\
&\subseteq& \{ (t_1,\ldots,t_{2k}) \in [0,1]^{2k} : (t_1+\cdots+t_k,t_{k+1}+\cdots+t_{2k}) \in \mathcal{D} \}, \label{R+R<D}
\end{eqnarray}
and let $f:[0,1]^{2k} \to \mathbb{R}$ be a step function supported on $(1-\varepsilon) \cdot \mathcal{R}$, and let
$$ w_f(m,n;\mathcal{G},\mathcal{H}) = (\ln{X})^{4k} \left( \sum_{\substack{d_i|n+h_i 1 \leq i \leq k \\ e_i|m-h_i 1 \leq i \leq k}} \prod_{i=1}^{k} \mu(d_i) \mu(e_i) F\left(\frac{\ln{d_1}}{\ln{X}},\ldots,\frac{\ln{d_k}}{\ln{X}},\frac{\ln{e_1}}{\ln{X}},\ldots,\frac{\ln{e_k}}{\ln{X}}\right) \right)^2 $$
where
$$ F(t_1,\ldots,t_{2k}) = \int_{t_1}^{1} \int_{t_2}^{1} \cdots \int_{t_{2k}}^{1} f(u_1,\ldots,u_{2k}) du_1 \cdots du_{2k}. $$

Then, for all even integers $N \in (X/2,X]$ with $O(X(\ln{X})^{-A})$ exceptions, the following hold:

For any integer $h,g$ with $|h|,|g| \leq (\ln{X})^3$, we have
\begin{equation} \label{S1 If}\sum_{m+n=N} w_f(m,n;\mathcal{G},\mathcal{H}) = N \left(\mathfrak{S}(\mathcal{G},\mathcal{H};N) I_{2k}(f) + o(1)\right) (\ln{X})^{2k} 
\end{equation} 
\begin{multline} \label{S2 Jf}
\sum_{m+n=N} \Lambda(m+g) \Lambda(n+h) w_f(m,n;\mathcal{G},\mathcal{H}) \\
=
\begin{cases}
N \left(\mathfrak{S}(\mathcal{G},\mathcal{H};N) J_{2k}^{(\ell,\ell')}(f) + o(1)\right) (\ln{X})^{2k+2} & g = g_{\ell} \in \mathcal{G}, h = h_{\ell'} \in \mathcal{H} \\
N \left(\mathfrak{S}(\mathcal{G} \cup \{g\},\mathcal{H} \cup \{h\};N) I_{2k}(f) + o(1)\right) (\ln{X})^{2k} & g \notin \mathcal{G}, h \notin \mathcal{H} \\
\end{cases}
\end{multline}
Here
$$ I_{2k}(f) = \int_{0}^{1} \cdots \int_{0}^{1} f(t_1,\ldots,t_{2k})^2 dt_1 \cdots dt_{2k}. $$
$$ J^{(\ell,\ell')}_{2k}(f) = \int_{0}^{1} \cdots \int_{0}^{1} \left(\int_{0}^{1} \int_{0}^{1} f(s_1,\ldots,s_k,t_1,\ldots,t_k) ds_{\ell} dt_{\ell'} \right)^2 ds_1 \cdots ds_k dt_1 \cdots dt_k. $$

\end{lem}

\begin{proof}
We prove only \eqref{S2 Jf} with $g = g_{1} \in \mathcal{G}, h = h_{1} \in \mathcal{H}$. Other cases follow in a similar manner. 

The function $f$ admits the expression
$$ f(s_1,\ldots,s_k,t_1,\ldots,t_k) = \sum_{j=1}^{J} c_j \prod_{i=1}^{k} \1_{[0,\alpha^{(j)}_i)}(s_i) \1_{[0,\beta^{(j)}_i)}(t_i) $$
for some
$$ J \in \mathbb{N}, (c_j)_{j=1}^{J} \in \mathbb{R}^J, ((\alpha^{(j)}_1,\ldots,\alpha^{(j)}_k,\beta^{(j)}_1,\ldots,\beta^{(j)}_k))_{j=1}^{J} \in ((1-\varepsilon) \cdot \mathcal{R})^J, $$
by the definition. By direct calculation, we have
\begin{equation} \label{J11} J^{(1,1)}_{2k}(f) = \sum_{j=1}^{J} \sum_{j'=1}^{J} c_j c_{j'} \alpha^{(j)}_1 \alpha^{(j')}_1 \beta^{(j)}_1 \beta^{(j')}_1 \prod_{i=2}^{k} \min(\alpha^{(j)}_i,\alpha^{(j')}_i) \min(\beta^{(j)}_i,\beta^{(j')}_i). \end{equation}
Since
$$ \int_{t}^{1} \1_{[0,\alpha)}(u) du = (\alpha - t) \1_{[0,\alpha)}(t), \quad \Lambda_{X^{\alpha}}(n) = \ln{X} \sum_{d|n} \mu(d) \left(\alpha - \frac{\ln{d}}{\ln{X}}\right) \1_{[0,\alpha)}\left(\frac{\ln{d}}{\ln{X}}\right), $$
we see that
$$ w_f(m,n;\mathcal{G},\mathcal{H}) = \left( \sum_{j=1}^{J} c_j \prod_{i=1}^{k} \Lambda_{X^{\alpha^{(j)}_i}}(m+g_i) \Lambda_{X^{\beta^{(j)}_i}}(n+h_i) \right)^2. $$
Hence, the left-hand side of \eqref{S1 If} becomes
$$ \sum_{j=1}^{J} \sum_{j'=1}^{J} c_j c_{j'} \sum_{m+n=N} \Lambda(m+g_1) \Lambda(n+h_1) \prod_{i=1}^{k} \Lambda_{X^{\alpha^{(j)}_i}}(m+g_i) \Lambda_{X^{\alpha^{(j)}_i}}(n+h_i) \Lambda_{X^{\alpha^{(j')}_i}}(m+g_i) \Lambda_{X^{\alpha^{(j')}_i}}(n+h_i). $$
By Lemma \ref{mainlem} \eqref{S2goal1} and \eqref{J11}, we obtain \eqref{S2 Jf}, since \eqref{R+R<D} implies
\begin{equation} \label{mainlem_a+b<D 2} \left( \sum_{i=2}^{k} (\alpha^{(j)}_i+\alpha^{(j')}_i), \sum_{i=2}^{k} (\beta^{(j)}_i+\beta^{(j')}_i) \right) \in (1-\varepsilon) \cdot \mathcal{D}. \end{equation}
for each $j,j' \leq J$. 
\end{proof}

\begin{rem}
One can instead use a weight of the form
$$ w_f(m,n;\mathcal{G},\mathcal{H}) = \left( \sum_{\substack{q_i,r_i}} \prod_{i=1}^{k} \frac{\mu(q_i) \mu(r_i)}{\varphi(q_i) \varphi(r_i)} c_{q_i}(m+g_i) c_{r_i}(n+h_i) f\left(\frac{\ln{q_1}}{\ln{X}},\ldots,\frac{\ln{q_k}}{\ln{X}},\frac{\ln{r_1}}{\ln{X}},\ldots,\frac{\ln{r_k}}{\ln{X}}\right) \right)^2, $$
which is closer to Maynard's choice of $\lambda$. To evaluate this, we only need an analogue of Lemma \ref{GYlem} in which $\Lambda_{Q,\text{SEL}}$ is replaced by $\Lambda_{Q,\text{HB}}$. 
This seems easier than the proof of Lemma \ref{GYlem}, which relies on a multidimensional version of a Perron-type formula. 
However, it is still too messy to present in full here. 
We briefly outline the idea. 

If $\varepsilon_i=0$ for all $i$, this follows from a classical result, due to Hardy-Littlewood, on singular series of $k$-prime tuple conjecture. See \cite{montgomery2004primes} Section 2 for simplification. For the general case, by expanding $\Lambda_{Q,\text{HB}}$, we encounter the cross term $c_{q}(n) c_{q'}(n)$ from the terms with $\varepsilon_i=1$. This can be rewritten as a single term using the identity
$$ c_{q}(n) c_{q'}(n) = \sum_{r|t} c_{ss'r}(n) \varphi(t/r) \psi(r) $$
where $q,q'$ are square-free integers with $(q,q')=t,q=st,q'=s't$, and $\psi(n)=n \prod_{p|n} (1 - \frac{2}{p})$. 
This reduces the problem to a similar case where $\varepsilon_i=0$. 
\end{rem}

\subsection{The limitations of evaluating $\Xi,\Xi^{*}$}

We return to the expression
$$ S = \sum_{\substack{m+n=N \\ N^{1-\epsilon} \leq m,n}} \sum_{\substack{1 \leq h_1,\ldots,h_k \leq H \\ h_i \text{ distinct}}} \left( \sum_{h \leq H} \Lambda'(m-h) \Lambda'(n+h) - (\ln{N})^2 \right) w(m,n;\mathcal{H}), $$
where $H=\lambda \mathfrak{S}(N)^{-1} (\ln{N})^2$. 
We take $w(m,n;\mathcal{H}) = w_f(m,n;\mathcal{H}) \coloneq w_f(m,n;-\mathcal{H},\mathcal{H})$, where the function $f$ satisfies the condition of Lemma \ref{gpymt mainlem}. We easily see that
\begin{eqnarray*}
S = \sum_{\substack{1 \leq h_1,\ldots,h_k \leq H \\ h_i \text{ distinct}}} \sum_{\substack{h \in \mathcal{H}}} \sum_{m+n=N} \Lambda(m+h) \Lambda(n-h) w_f(m,n;\mathcal{H}) \\
+ \sum_{\substack{1 \leq h_1,\ldots,h_k \leq H \\ h_i \text{ distinct}}} \sum_{\substack{h \leq H \\ h \notin \mathcal{H}}} \sum_{m+n=N} \Lambda(m+h) \Lambda(n-h) w_f(m,n;\mathcal{H}) \\
- (\ln{N})^2 \sum_{\substack{1 \leq h_1,\ldots,h_k \leq H \\ h_i \text{ distinct}}} \sum_{m+n=N} w_f(m,n;\mathcal{H})
+ O(N). 
\end{eqnarray*}
By Lemma \ref{gpymt mainlem} and Lemma \ref{S(H) av}, we have
\begin{eqnarray*}
S &=& N (\ln{X})^{2k} \left( (\ln{X})^2 \sum_{\ell=1}^{k} J^{(\ell,\ell)}_{2k}(f) \sum_{\substack{1 \leq h_1,\ldots,h_k \leq H \\ h_i \text{ distinct}}} \mathfrak{S}(\mathcal{H};N) \right. \\ && \left. + I_{2k}(f) \sum_{\substack{1 \leq h_1,\ldots,h_{k+1} \leq H \\ h_i \text{ distinct}}} \mathfrak{S}(\mathcal{H};N) - I_{2k}(f) (\ln{N})^2 \sum_{\substack{1 \leq h_1,\ldots,h_k \leq H \\ h_i \text{ distinct}}} \mathfrak{S}(\mathcal{H};N) + o(H^k (\ln{X})^2 + H^{k+1}) \right), \\
 &=& N (\ln{X})^{2k+2} (H\mathfrak{S}(N))^{k} \left( \sum_{\ell=1}^{k} J^{(\ell,\ell)}_{2k}(f) +  \gamma I_{2k}(f) - I_{2k}(f) + o(1) \right).
\end{eqnarray*}
Note that $\mathfrak{S}(N) \gg 1$. 
This is positive if 
$$ \gamma > 1 - \frac{\sum_{\ell=1}^{k} J^{(\ell,\ell)}_{2k}(f)}{I_{2k}(f)}. $$

If we consider \eqref{bounded S}, this becomes
$$ = N (\ln{X})^{2k+2} \left( \mathfrak{S}(\mathcal{H};N) \left( \sum_{\ell=1}^{k} J^{(\ell,\ell)}_{2k}(f) - I_{2k}(f) \right)+ o(1) \right), $$
by a similar calculation. Note that if $\mathfrak{S}(\mathcal{H};N)\neq 0$, then $\mathfrak{S}(\mathcal{H};N) \geq C_k$ for some constant $C_k>0$ depending only on $k$. 

Therefore, we obtain the following proposition. 
\begin{prop}
Let $M_k$ be the supremum
$$ M_k(\mathcal{D}) = \sup_{f} \frac{\sum_{\ell=1}^{k} J^{(\ell,\ell)}_{2k}(f)}{I_{2k}(f)}. $$
over all piecewise continuous functions $f$ whose support is on a closed set $\mathcal{R} \subseteq [0,1]^{2k}$ which satisfies
\begin{eqnarray*}
\label{R+R(D} \mathcal{R}+\mathcal{R} &=& \{ (s_1+t_1,\ldots,s_{2k}+t_{2k}) : (s_1,\ldots,s_{2k}),(t_1,\ldots,t_{2k}) \in \mathcal{R} \} \\
&\subseteq& \{ (t_1,\ldots,t_{2k}) \in [0,1]^{2k} : (t_1+\cdots+t_k,t_{k+1}+\cdots+t_{2k}) \in \mathcal{D} \} \eqcolon \mathcal{D}_k. 
\end{eqnarray*}
Then, we have
$$ \Xi \leq \inf_{k \geq 1} \max(1 - M_k(\mathcal{D}),0) $$

Furthermore, if there exists $k \geq 2$ and $\mathcal{H}=\{h_1,\ldots,h_k\}$ such that $M_k(\mathcal{D})>1$ and $\mathfrak{S}(\mathcal{H};N) \neq 0$, then we have
$$ \Xi^{*} \leq \max_{h,h' \in \mathcal{H}} (h-h'). $$
\end{prop}

One can verify that 
$$ M_1 = \sup_{\substack{\mathcal{R} \subseteq [0,1]^2: \text{ closed} \\ \mathcal{R}+\mathcal{R} \subseteq \mathcal{D}}} \mathrm{meas}(\mathcal{R}) $$
by Cauchy-Schwarz inequality, and 
$$M_1 \leq M_2 \leq M_3 \cdots$$
by taking
$$ f_k(s_1,\ldots,s_k,t_1,\ldots,t_k) = f_{k-1}\left((1-\varepsilon)s_1,\ldots,(1-\varepsilon)s_{k-1},(1-\varepsilon)t_1,\ldots,(1-\varepsilon)t_{k-1}\right) \1_{s_k+t_{k} \leq \varepsilon}. $$
and letting $\varepsilon \to 0$.

We hope that $M_k \to \infty$ as $k \to \infty$, or at least that $M_k$ becomes moderately large for some $k$. We are currently unable to provide any meaningful lower bound for $M_k$. Nevertheless, following the Maynard-type argument and solving the arising Euler-Lagrange equation suggests the choice
$$ f(s_1,\ldots,s_k,t_1,\ldots,t_k) = \1_{(s_1,\ldots,t_k) \in \mathcal{R}} \prod_{i=1}^{k} \frac{1}{1 + U(s_i + t_i)} $$
for some $U>0$. Unfortunately, combined with the Cauchy-Schwarz inequality (\cite{polymath2014variants} Lemma 6.1), we have the following result. 

\begin{prop}
For any $k$, we have
\begin{equation} \label{Mk bound} M_k(\mathcal{D}) \leq \left( \frac{1}{2} \max_{(s,t) \in \mathcal{D}} (s+t) \right)^2. \end{equation}
\end{prop}

\begin{proof}
Let $f$ be a piecewise continuous function supported on $\frac{1}{2} \cdot \mathcal{D}_k$ and denote
$$ M_k(f) = \frac{\sum_{\ell=1}^{k} J^{(\ell,\ell)}_{2k}(f)}{I_{2k}(f)}. $$
Since $2\cdot \mathcal{R} \subseteq \mathcal{R} + \mathcal{R}$, \eqref{R+R<D} implies $\mathcal{R} \subseteq \frac{1}{2} \cdot \mathcal{D}_k$. Thus, it suffice to show that $M_k(f) \leq h^2 $, where $h=\frac{1}{2} \max_{(s,t) \in \mathcal{D}} (s+t)$.

We omit the limits of integration when they are understood to be over $[0,1]$. 

Let $g:[0,1]^{2k} \to \mathbb{R}$ be a piecewise smooth function which is nonzero on $\frac{1}{2} \cdot \mathcal{D}_k$. Writing $f = (f g^{-1/2}) \cdot (g^{1/2}),$
and applying the Cauchy–Schwarz inequality, we have
\begin{multline*}
\left( \iint f(t_1,\ldots,t_{2k}) \, dt_{\ell} dt_{\ell+k} \right)^2 \leq \left( \iint f(t_1,\ldots,t_{2k})^2 g(t_1,\ldots,t_{2k})^{-1} \, dt_{\ell} dt_{\ell+k} \right), \\
\times \left( \iint g(t_1,\ldots,t'_{\ell},\ldots,t'_{k+\ell},\ldots,t_{2k}) \, dt'_{\ell} dt'_{k+\ell} \right).
\end{multline*}
Therefore, we have
\begin{align*}
\sum_{\ell=1}^{k} J^{(\ell,\ell)}_{2k}(f) &\leq \int \cdots \int f(t_1,\ldots,t_{2k})^2 g(t_1,\ldots,t_{2k})^{-1} \\
&\quad \times \sum_{\ell=1}^{k} \left( \iint g(t_1,\ldots,t'_{\ell},\ldots,t'_{\ell+k},\ldots,t_{2k}) \, dt_{\ell}' dt_{k+\ell}' \right) dt_1 \cdots dt_{2k} \\
&\leq \left( \sup_{(t_1,\ldots,t_{2k}) \in \frac{1}{2} \cdot \mathcal{D}_k} g(t_1,\ldots,t_{2k})^{-1} \sum_{\ell=1}^{k}  \left( \iint g(t_1,\ldots,t'_{\ell},\ldots,t'_{\ell+k},\ldots,t_{2k}) \, dt'_{\ell} dt'_{k+\ell} \right) \right) \\
&\quad \times \int \cdots \int f(t_1,\ldots,t_{2k})^2 \, dt_1 \cdots dt_{2k}.
\end{align*}
Comparing with the definition of $I_{k}$, we obtain
$$ M_k(f) \leq \sup_{(t_1,\ldots,t_{2k}) \in \frac{1}{2} \cdot \mathcal{D}_k} \frac{\sum_{\ell=1}^{k} \iint g(t_1,\ldots,t'_{\ell},\ldots,t'_{\ell+k},\ldots,t_{2k}) \, dt'_{\ell} dt'_{k+\ell}}{g(t_1,\ldots,t_{2k})}. $$

Suppose we take 
$$g(t_1,\ldots,t_{2k}) = \prod_{\ell=1}^{k} G(t_{\ell}, t_{\ell+k}) $$
for some function $G$ satisfying the conditions for $g$. Then we have
$$ M_k(f) \leq \iint G(s,t) \, ds dt \cdot \sup_{(t_1,\ldots,t_{2k}) \in \frac{1}{2} \cdot \mathcal{D}_k} \sum_{\ell=1}^{k} \frac{1}{G(t_{\ell},t_{\ell+k})}. $$ 
We now take
$$ G(s,t) = \1_{s+t\leq h} \cdot \frac{1}{1 + U(s + t)}. $$
Then $g$ is piecewise smooth and nonzero on $\frac{1}{2} \cdot \mathcal{D}_k$. We have
\begin{eqnarray*}
\iint G(s,t) ds dt &=& \int_0^h \int_0^{h-t} \frac{1}{1 + U(s + t)} ds dt, \\
&=& \frac{1}{U} \int_0^h \left( \ln(1 + hU) - \ln(1 + tU) \right) dt, \\
&=& \frac{h}{U} - \frac{1}{U^2} \ln(1 + hU) \leq \frac{h}{U}.
\end{eqnarray*}
Moreover, for any $(t_1,\ldots,t_{2k}) \in \frac{1}{2} \cdot \mathcal{D}_k$, we have
$$
\sum_{\ell=1}^{k} \frac{1}{G(t_{\ell},t_{\ell+k})} = k + U \sum_{i=1}^{2k} t_i \leq k + Uh.
$$
Putting everything together, we obtain
$$
M_k(,f) \leq \frac{k + hU}{U} h.
$$
Letting $U \to \infty$, we deduce that
$$
M_k(\mathcal{D}) \leq h^2 = \left( \frac{1}{2} \max_{(s,t) \in \mathcal{D}} (s+t) \right)^2.
$$
\end{proof}

Therefore, we conclude that we cannot obtain $\Xi \leq 1 - \left( \frac{1}{2} (\frac{1}{2}+\frac{1}{3}) \right)^2 = \frac{119}{144} = 0.82638...$ using the current method. Furthermore, even under the strong assumption $\mathcal{D}=[0,1] \times [0,1]$, it is not sufficient to prove the existence of $\Xi^{*}$. 
It is remarkable that, for many typical choices of $\mathcal{D}$ (though not always), one has
$$ \left( \frac{1}{2} \max_{(s,t) \in \mathcal{D}} (s+t) \right)^2 \leq \mathrm{meas}(\mathcal{D}). $$

The current method uses a smaller range of the level of distribution of $\{(m,n):m+n=N\}$, compared to the Bombieri-Davenport method. A suitable modification of the Maynard-Tao method, which fully exploits the level of distribution, might overcome the Bombieri-Davenport method. One should consider a more general setting such as
\begin{equation}
\label{general set l}
\sum_{\substack{d_i,e_i \\ d_i|m-h_i \\ e_i|n+h_i}} \lambda_{d_1,\ldots,d_k,e_1,\ldots,e_k} + \sum_{\substack{1 \leq \ell \leq k \\ 1 \leq \ell' \leq k}} \Lambda(m-h_\ell) \Lambda(n+h_{\ell'}) \sum_{\substack{d_i,e_i \\ d_i|m-h_i \\ e_i|n+h_i}} \lambda^{(\ell,\ell')}_{d_1,\ldots,d_k,e_1,\ldots,e_k} \leq \1_{\exists i \neq \exists j, n+h_i,n+h_j \in \mathbb{P}(N)}. 
\end{equation}

\begin{que}
Is it possible to establish the existence of $\Xi^*$, by considering \eqref{general set l} with $\mathcal{D}=[0,1] \times [0,1]$?
\end{que}

We note that there is no parity barrier for this problem when $k \geq 5$; this follows from a generalization of an argument of Zeb Brady and Terence Tao
\footnote{
See Tao's blog post \cite{TaoBlog}. 
For $k=3$, one can (albeit non-rigorously) verify the absence of the parity barrier using the non-negative weight 
$$ w(n) = q(\lambda(m-h_1),\lambda(m-h_2),\lambda(m-h_3)) q(\lambda(n+h_3),\lambda(n+h_2),\lambda(n+h_1)) $$
where $\lambda$ is the Liouville function and $q(x,y,z)=(1+x y-x z - y z) = \frac{1}{2} ((1+x)(1+y)(1-z)+(1-x)(1-y)(1+z))$. The case $k=4$ is more complicated. 
See Section 8 of \cite{polymath2014variants} for a similar argument. 
}.

\subsection{Proof of Theorem \ref{Xi**<}}

Let $\mathcal{H}$ be a set of $k$ integers, and consider
$$ S = \sum_{\substack{m+n=N \\ N^{1-\epsilon} \leq m,n}} \left( \sum_{g,h \in \mathcal{H}} \Lambda'(m-g) \Lambda'(n+h) - (\ln{N})^2 \right) w_f(m,n;\mathcal{H}). $$
If $S>0$, then there exist $m+n=N, g,h,g',h' \in \mathcal{H}, (g,h) \neq (g',h')$ such that $p_1=m-g, p_2=n+h, p_3=m-g', p_4=n+h'$ are prime. Since $|N-p_1-p_2|,|N-p_3-p_4|,|p_1-p_3|,|p_2-p_4| \leq \max_{h,h' \in \mathcal{H}} (h-h')$, and $p_1 \neq p_3$ or $p_2 \neq p_4$, we see that
$$ \min_{\substack{p,p' \in \mathbb{P}_H(N) \\ N^{1-\epsilon} < p'<p}} (p-p') \leq H $$
with $H=\max_{h,h' \in \mathcal{H}} (h-h')$. By Lemma \ref{gpymt mainlem}, we have
$$ S = N (\ln{X})^{2k+2} \left( \mathfrak{S}(\mathcal{H};N) \left( \sum_{\ell=1}^{k} \sum_{\ell'=1}^{k}  J_{2k}^{(\ell,\ell')}(f) - I_{2k}(f) \right) + o(1) \right)$$
Thus, we have $S>0$ if $\mathfrak{S}(\mathcal{H};N)>0$ and 
$$ \frac{\sum_{\ell=1}^{k} \sum_{\ell'=1}^{k} J_{2k}^{(\ell,\ell')}(f)}{I_{2k}(f)} > 1. $$
\begin{dfn}\label{def st ad}
We say that a set of $k$ integers $\mathcal{H}=\{h_1,\ldots,h_k\}$ is Goldbach-admissible if, for every even integer $N$, $\mathcal{H} \cup (\mathcal{H}-N)$ is an admissible tuple, that is
$$ |\mathcal{H} \cup (\mathcal{H}-N)/p\mathbb{Z}| <p $$
holds for every prime $p$.
\end{dfn}
Note that $\mathcal{H}$ is Goldbach-admissible if and only if $\mathfrak{S}(\mathcal{H};N)>0$ holds for all even integers $N$. 

For example, 
\begin{equation} \label{ad_1} \mathcal{H} = \{ p^2_{\pi(2k)+1}, \ldots, p^2_{\pi(2k)+k} \}, \end{equation}
is a Goldbach-admissible $k$-tuple. Here $p_i$ denotes the $i$-th prime. 
This follows from the fact that the number of quadratic residues modulo $p$ is $\frac{p-1}{2}$ and simple inequalities $|\mathcal{H} \cup (\mathcal{H}-N)/p\mathbb{Z}| \leq 2|\mathcal{H}/p\mathbb{Z}|, p_{\pi(i)+1}>i$.
This should be compared with constructions of admissible tuple in \cite{maynard2015small} and \cite{polymath2014variants}.

\begin{comment}
We can check this from a simple inequality $|\mathcal{H} \cup (\mathcal{H}-N)/p\mathbb{Z}| \leq 2|\mathcal{H}/p\mathbb{Z}|$. 
If $p>2k$ then $2|\mathcal{H}/p\mathbb{Z}| \leq 2|\mathcal{H}|=2k < p$. If $p \geq 2k$ and thus $p < p_{\pi(2k)+1}$ then $2|\mathcal{H}/p\mathbb{Z}| \leq p-1$ since the number of quadratic residues modulo $p$ is $\frac{p-1}{2}$. 
\end{comment}

Finding a narrow Goldbach-admissible tuple will improve the final result. 
Let $H(k)$ denote the minimal length $\max_{h,h' \in \mathcal{H}} (h-h')$ of a Goldbach-admissible $k$-tuple $\mathcal{H}$. 
\eqref{ad_1} gives
$$ H(k) \leq p^2_{\pi(2k)+k} - p^2_{\pi(2k)+1} = (k \ln{k})^2 (1+o(1)). $$
However, we can remove the $(\ln{k})^2$ factor with a more careful calculation.  

Let $\mathcal{H}$ be a finite set of integers, and let  
\begin{equation} \label{Omega def} \Omega(p) = (\mathbb{Z}/p\mathbb{Z}) \setminus (\mathcal{H}/p\mathbb{Z}). \end{equation}
for a prime $p$.

\begin{lem}\label{disc ad}
The set $\mathcal{H}$ is Goldbach-admissible if and only if $|\Omega(2)|=1$, and for every $p>2$, 
\begin{equation} \label{sumset} \Omega(p) - \Omega(p) \coloneqq \{ a-b : a,b \in \Omega(p) \} = \mathbb{Z}/p\mathbb{Z}. \end{equation}
\end{lem}

\begin{proof}
Let $p>2$. For an even integer $N$, the condition
$$ ( \mathcal{H} \cup (\mathcal{H} - N) ) /p\mathbb{Z} \subsetneq \mathbb{Z}/p\mathbb{Z}, $$
is equivalent to
$$ \Omega(p) \cap (\Omega(p)-N) \neq \emptyset. $$
The requirement that this holds for all even integers $N$ is equivalent to
$$ \Omega(p) - \Omega(p) = \{ N \pmod p : N \text{ even} \} = \mathbb{Z}/p\mathbb{Z}. $$
\end{proof}

\begin{lem}\label{Hk<}
For any $H\geq 1$, 
\begin{equation} \label{ad_2} \mathcal{H} = \{ h^2 : 1 \leq h \leq H, (h,30)=1 \}, \end{equation}
is Goldbach-admissible. Thus we have
\begin{equation} \label{ad_2_asy} H(k) \leq \left( \frac{15 k}{4} \right)^2 + O(k). \end{equation}
In particular, we have
\begin{equation} \label{ad_2_2000} H(2000) \leq (7500)^2 = 56250000. \end{equation}
\end{lem}
\begin{proof}
Let $\left( \frac{\cdot}{p} \right)$ denote Legendre symbol. We have
\begin{equation*}
\Omega(p) \supseteq 
\begin{cases}
\{0\} & p=2 \\
\{0,2\} & p=3 \\
\{0,2,3\} & p=5 \\
\{a \in \mathbb{Z}/p\mathbb{Z}: \left( \frac{a}{p} \right) = -1 \} & \text{otherwise.}
\end{cases}
\end{equation*}
This immediately gives $|\Omega(2)|=1$ and \eqref{sumset} for $p=3,5$. 

For $p>5$ and fixed $n \in \mathbb{Z}/p\mathbb{Z}$, the number of solutions $(a,b) \in (\mathbb{Z}/p\mathbb{Z})^2$ of equation $a-b=n$ with $\left( \frac{a}{p} \right),\left( \frac{b}{p} \right)=-1$ is equal to
\begin{equation} \label{a-b=n} \sum_{a=1}^{p} \1_{\left( \frac{a}{p} \right)=-1} \1_{\left( \frac{a-n}{p} \right)=-1} = \frac{1}{4} \sum_{b=1}^{p} \left( 1 - \left( \frac{a}{p} \right) - \1_{p|a} \right) \left( 1 - \left( \frac{a-n}{p} \right) - \1_{p|(a-n)} \right). \end{equation}
Recalling the well-known identity
$$ \sum_{a=1}^{p} \left( \frac{a}{p} \right) \left( \frac{a-n}{p} \right) = p \1_{p|n} - 1, $$
we see that \eqref{a-b=n} equals $(p+p\1_{p|n}+\left( \frac{n}{p} \right)+\left( \frac{-n}{p} \right)-3+\1_{p|n})/4\geq (p-5)/4 >0$. 

The bounds \eqref{ad_2_asy} and \eqref{ad_2_2000} follow from $\#\{1 \leq h \leq H : (h,30)=1 \}=\frac{\varphi(30)}{30} H + O(1)$ and $\#\{1 \leq h \leq 7500 : (h,30)=1 \}=2000$. 
\end{proof}

Let $M^2_k(\mathcal{D})$ be the supremum
$$ M^2_k(\mathcal{D}) = \sup \frac{\sum_{\ell=1}^{k} \sum_{\ell'=1}^{k} J^{(\ell,\ell')}_{2k}(f)}{I_{2k}(f)} $$
over all piecewise continuous functions $f$ whose support is on a closed set $\mathcal{R} \subseteq [0,1]^{2k}$ satisfying $\mathcal{R}+\mathcal{R} \subseteq \mathcal{D}_k$. 
\begin{lem}\label{maynard vari}
Let $k=2000$. There exists a piecewise smooth function $F:[0,1]^k \to \mathbb{R}$ whose support is on $\{ (t_1,\ldots,t_k) \in [0,1]^k : t_1+\cdots+t_k \leq 1\}$ and 
$$ \sum_{\ell=1}^{k} \int_{0}^{1} \cdots \int_{0}^{1} \left( \int_{0}^{1} F(t_1,\ldots,t_k) dt_{\ell} \right)^2 dt_1 \ldots dt_k > 5 \int_{0}^{1} \cdots \int_{0}^{1} F(t_1,\ldots,t_k)^2 dt_1 \ldots dt_k > 0 $$
holds. 
\end{lem}
\begin{proof}
By Theorem 6.7 in \cite{polymath2014variants} with the choice of parameters
$$ c = \frac{0.96}{\ln{k}}, \quad T = \frac{0.97}{\ln{k}}, \quad \tau = 1 - k \mu, $$
and numerical calculation using Mathematica, we see that there exists a piecewise continuous function which satisfies
$$ \frac{\sum_{\ell=1}^{k} \int_{0}^{1} \cdots \int_{0}^{1} \left( \int_{0}^{1} F(t_1,\ldots,t_k) dt_{\ell} \right)^2 dt_1 \ldots dt_k}{\int_{0}^{1} \cdots \int_{0}^{1} F(t_1,\ldots,t_k)^2 dt_1 \ldots dt_k} \geq 5.00958\ldots $$
\end{proof}

\begin{lem}\label{vari}
We have
$$ M^2_{2000}(\mathcal{D}) > 1 $$
\end{lem}

\begin{proof}
Set $k=2000$. 
Let $F$ be a function satisfying Lemma \ref{maynard vari}, and define
$$ f(t_1,\ldots,t_{2k}) = F(4t_1,\ldots,4t_k) F(6t_{k+1},\ldots,6t_{2k}). $$
Clearly, $f$ is a piecewise smooth function with support contained in 
$$ \mathcal{R} \coloneqq \{ (t_1,\ldots,t_{2k}) \in [0,1]^{2k} : (t_1+\ldots+t_k,t_{k+1}+\ldots+t_{k2}) \in [0,1/4] \times [0,1/6] \} $$
which satisfies $\mathcal{R} + \mathcal{R} = 2 \cdot \mathcal{R} \subseteq \mathcal{D}_k$. 

A straightforward computation shows that
\begin{multline*}
J^{(\ell,\ell')}_{k}(f) = \frac{1}{4^{k+1} \cdot 6^{k+1}} \int_{0}^{1} \cdots \int_{0}^{1} \left( \int_{0}^{1} F(t_1,\ldots,t_k) dt_{\ell} \right)^2 dt_1 \ldots dt_k \\
 \times \int_{0}^{1} \cdots \int_{0}^{1} \left( \int_{0}^{1} F(t_{k+1},\ldots,t_{2k}) dt_{k+\ell'} \right)^2 dt_{k+1} \ldots dt_{2k},
\end{multline*}
and
$$ I_k(f) = \frac{1}{4^{k} \cdot 6^{k}} \int_{0}^{1} \cdots \int_{0}^{1} F(t_1,\ldots,t_k)^2 dt_1 \ldots dt_k \times \int_{0}^{1} \cdots \int_{0}^{1} F(t_{k+1},\ldots,t_{2k})^2 dt_{k+1} \ldots dt_{2k}. $$

Thus, we have
\begin{eqnarray*}
M^2_k(\mathcal{D}) &\geq& \frac{\sum_{\ell=1}^{k} \sum_{\ell'=1}^{k} J^{(\ell,\ell')}_{k}(f)}{I_{k}(f)}, \\
&=& \left( \frac{1}{\sqrt{4 \cdot 6}} \frac{\sum_{\ell=1}^{k} \int_{0}^{1} \cdots \int_{0}^{1} \left( \int_{0}^{1} F(t_1,\ldots,t_k) dt_{\ell} \right)^2 dt_1 \ldots dt_k}{ \int_{0}^{1} \cdots \int_{0}^{1} F(t_1,\ldots,t_k)^2 dt_1 \ldots dt_k} \right)^2 > 1.
\end{eqnarray*}
\end{proof}

By Lemma \ref{vari} and Lemma \ref{Hk<}, we conclude that $S>N^{1/2+\varepsilon}$ occurs when
\begin{eqnarray*}
H &=& \max_{h,h' \in \mathcal{H}} (h-h') \leq H(2000) \leq 56250000.
\end{eqnarray*}

Needless to say, there is much room for improvement in the above bound. 
We did not apply a more elaborate method for finding an optimal $F$, which uses quadratic optimization and requires enormous computation, when proving Lemma \ref{maynard vari}. It seems that the first $k$ satisfying the condition of Lemma \ref{maynard vari} (with $5$ replaced by $\sqrt{24}$) lies in the range $100 \leq k \leq 200$. 
We remark that we have only used the level of distribution $[0,1/2] \times [0,1/3]$. Further improvements may be possible by fully exploiting $\mathcal{D}$. 
For example, one may take $f$ of the form
$$ f(s_1,\ldots,s_{k},t_1,\ldots,t_{k}) = \int_{(u,v) \in \mathcal{R}_2 \cap \mathbb{R}^2_{>0}} g(u,v) F\left(\frac{s_1}{u},\ldots,\frac{s_k}{u}\right) F\left(\frac{t_1}{v},\ldots,\frac{t_k}{v}\right) du dv $$
where $\mathcal{R}_2$ is a closed set satisfying $\mathcal{R}_2+\mathcal{R}_2 \subseteq \mathcal{D}_2$ (e.g., $\mathcal{R}_2 = \{(u,v) \in [0,1/4] \times [0,1/4] : u+v \leq 1/3 \}$), and $g$ is an arbitrary smooth function.

We also note that computer-assisted calculation of $H(k)$ may improve the result. 
Alternatively, it may be better to extend the definition of a Goldbach-admissible tuple. 
For example, one can allow $\mathcal{H}$ to depend on $N \pmod W$ for some $W=\prod_{p \ll_k 1} p$ by proving Theorem \ref{LL level} for each arithmetic progression modulo $W$. 
Furthermore, there is no reason to use the same tuple for the variables $m$ and $n$ in proving Theorem \ref{Xi**<}. 
Separately, the asymptotic behavior of $H(k)$ may be of independent interest. We briefly record a lower bound for $H(k)$.

\begin{prop}\label{st ad lower}
We have
\begin{equation} \label{ad bounds} H(k) \gg k^{4/3}. \end{equation}
\end{prop}

\begin{proof}
Definition \ref{def st ad} is invariant under shifts, thus we may assume the minimal element of $\mathcal{H}$ is $1$. The bound \eqref{ad bounds} is equivalent to 
$$ \mathcal{H} \subseteq [1,N], \text{ is Goldbach-admissible} \Rightarrow |\mathcal{H}| \ll N^{3/4}. $$

Let $z,N \geq 2$ and define 
$$ S([1,N] \cap \mathbb{Z},z;\Omega) \coloneqq \{ n \leq N : \forall p < z , n \pmod p \notin \Omega(p) \}. $$
where $\Omega(p)$ is defined in \eqref{Omega def}. We have $\mathcal{H} \subseteq S([1,N] \cap \mathbb{Z},z;\Omega)$. 

By Lemma \ref{disc ad}, we see that
\begin{equation} \label{O-O<O*O} p = |\Omega(p) - \Omega(p)| \leq |\Omega(p)|^2 \end{equation}
which implies $|\Omega(p)| \geq \sqrt{p}$ for all $p>2$. Note that $|\Omega(2)|=1$. 

The arithmetic large sieve inequality (see Corollary 4.3 in \cite{montgomery2006topics} or Chapter 2 in \cite{motohashi2005overview}) gives 
$$ |S([1,N] \cap \mathbb{Z},z;\Omega)| \leq \frac{1}{G(z,\Omega)} (N + 2z^2), $$
where
$$ G(z,\Omega) = \sum_{g < z} \mu(g)^2 \prod_{p|g} \frac{|\Omega(p)|}{p-|\Omega(p)|}. $$
The bound \eqref{O-O<O*O} gives
$$ G(z,\Omega) \geq \sum_{g < z} \mu(g)^2 \prod_{\substack{p|g \\ p>2}} \frac{\sqrt{p}}{p-\sqrt{p}} \geq \sum_{g < z} \mu(g)^2 \prod_{p|g} \frac{1}{\sqrt{p}} \geq \sum_{g < z} \frac{\mu(g)^2}{g^{1/2}} \gg z^{1/2}. $$
Thus we obtain
$$ |S([1,N] \cap \mathbb{Z},z;\Omega)| \ll Nz^{-1/2} + z^{3/2}. $$
Setting $z=N^{1/2}$, the proof is now complete. 
\end{proof}

\begin{que}
What is the value of $\lim_{k \to \infty} \frac{\ln{H(k)}}{\ln{k}}$?
\end{que}

It is worth noting that it is conjectured that the arithmetic large sieve inequality does not provide the correct order of magnitude unless the sieved set $S([1,N] \cap \mathbb{Z},z;\Omega)$ has some algebraic structure, at least when $|\Omega(p)| \geq \delta p$ holds for all $p$ and some fixed $\delta>0$. See Section 4.2 of \cite{helfgott2009small}.

\section*{Acknowledgments}

The author is grateful to H.~Mikawa and Y.~Suzuki for their encouragement. 
The author is also grateful to the organizers, Y.~Yasufuku and W.~Takeda, of the Hachioji Number Theory Seminar 2024, where the author learned of Maynard's paper. At this seminar, Suzuki served as the seminar organizer, and Mikawa suggested the problem. 

The author is also grateful to S.~Jain for bringing to the author's attention the result of Matom{\"a}ki and Shao on this subject.

\bibliographystyle{plain}
\nocite{*}
\bibliography{ref.bib}

\end{document}